\documentclass[12pt,reqno]{amsart}

\usepackage{XCharter}

\usepackage{amsmath, amssymb, amsthm}
\usepackage{mathtools}        
\usepackage{mathrsfs}         
\usepackage{stmaryrd}         

\usepackage{tikz}
\usetikzlibrary{cd, decorations.pathmorphing, arrows, positioning}

\usepackage[top=1in, bottom=1in, left=2.5cm, right=2.5cm]{geometry}
\usepackage{parskip}

\usepackage{microtype}

\usepackage[shortlabels]{enumitem}  

\usepackage{appendix}
\usepackage{verbatim}
\usepackage{comment}
\usepackage{graphicx}
\usepackage{subfiles}
\usepackage{chngcntr}

\usepackage[colorlinks, pagebackref]{hyperref}
\hypersetup{
  colorlinks = true,
  citecolor  = red,
  linkcolor  = blue,
  urlcolor   = cyan
}

\theoremstyle{definition}
\newtheorem{para}{}[subsection]

\theoremstyle{plain}
\newtheorem{thm}{Theorem}[section]
\newtheorem{lem}[thm]{Lemma}
\newtheorem{cor}[thm]{Corollary}
\newtheorem{prop}[thm]{Proposition}

\theoremstyle{definition}

\newtheorem{question}[thm]{Question}

\theoremstyle{remark}
\newtheorem{rem}[thm]{Remark}

\newcommand{\bbC}{\mathbb{C}}  \newcommand{\bbD}{\mathbb{D}}
  \newcommand{\bbF}{\mathbb{F}}

  \newcommand{\bbP}{\mathbb{P}}
\newcommand{\bbQ}{\mathbb{Q}}  \newcommand{\bbR}{\mathbb{R}}

  \newcommand{\bbZ}{\mathbb{Z}}

\newcommand{\frm}{\mathfrak{m}}

  \newcommand{\cH}{\mathcal{H}}
  
\newcommand{\cK}{\mathcal{K}}  
  
\newcommand{\cO}{\mathcal{O}}

  \newcommand{\rH}{\mathrm{H}}

\newcommand{\ol}[1]{\overline{#1}}

\newcommand{\Ql}{\bbQ_{\ell}}
\newcommand{\Zl}{\bbZ_{\ell}}

\DeclareMathOperator{\CH}{CH}

\DeclareMathOperator{\Hom}{Hom}

\DeclareMathOperator{\Spec}{Spec}

\DeclareMathOperator{\Tor}{Tor}

\newcommand{\Fl}{\bbF_{\ell}}                    
\newcommand{\Tl}{T_{\ell}}                       
\DeclareMathOperator{\Supp}{Supp}                
\DeclareMathOperator{\codim}{codim}              
\DeclareMathOperator{\Perv}{Perv}                
\DeclareMathOperator{\lgth}{length}              
\DeclareMathOperator{\SSupp}{SS}                 
\DeclareMathOperator{\CC}{CC}                    
\DeclareMathOperator{\FM}{FM}                    
\newcommand{\Dbc}{D^{b}_{c}}                     
\newcommand{\RHom}{R\Hom}                        
\newcommand{\Ldot}{\mathcal{L}}                  

\title{Artin vanishing along the \(\ell\)-adic tower of an abelian variety}

\author{Ashutosh Roy Choudhury}
\address{Department of Mathematics and Statistics, Boston University, Boston, MA 02215, USA}
\email{ashutosh@bu.edu}

\author{K. V. Shuddhodan}
\address{Department of Mathematics, University of Notre Dame, Notre Dame, IN 46556, USA}
\email{skadattu@nd.edu}
\date{September 28, 2026}
\subjclass[2020]{Primary 14F20; Secondary 14K05, 14C17}

\begin{document}

\begin{abstract}
Let \(A\) be an abelian variety of dimension \(g\) over an algebraically closed field and let
\(\ell\) be a prime invertible in the field. For every constructible \(\Fl\)-sheaf \(F\) on \(A\) we
prove that there is an integer \(e\), depending on \(F\), such that the pullback map
\([\ell^{e}]^{*}\colon\rH^{i}(A,[\ell^{n}]^{*}F)\to\rH^{i}(A,[\ell^{n+e}]^{*}F)\) is zero for all
\(n\geq0\) and all \(i>\dim\operatorname{Supp}F\). In particular
\(\varinjlim\limits_{n}\rH^{i}(A,[\ell^{n}]^{*}F)=0\) for \(i>\dim\operatorname{Supp}F\), which answers a
question of Bhatt--Schnell--Scholze.

Our methods also give sharp codimension estimates for the supports of the cohomology sheaves of the
Fourier--Mellin transform of a perverse sheaf. With \(\ell\)-adic coefficients, we remove the
arithmeticity hypothesis from the estimates of Esnault--Kerz. The corresponding estimates for the completed transform
with \(\Fl\)-coefficients are new and hold even when Hard Lefschetz fails.
\end{abstract}

\maketitle
\setcounter{tocdepth}{1}
\tableofcontents

\section{Introduction}
\label{sec:introduction}

Let \(k\) be an algebraically closed field, let \(\ell\) be a prime invertible in \(k\), and let \(A\) be
an abelian variety of dimension \(g\) over \(k\).

\begin{question}[Bhatt--Schnell--Scholze]
\label{question:BSS}
Given a constructible sheaf \(F\) of \(\Fl\)-vector spaces on \(A\), does the direct limit
\[
  \varinjlim_n\ \rH^{i}\bigl(A,[\ell^n]^{*}F\bigr),
\]
formed along the pullback maps of the tower of multiplication by \(\ell\), vanish for
\(i>\dim\Supp F\)?
\end{question}

Let \(A_{\infty}\) be the limit of the tower \(\cdots\to A\xrightarrow{\ell}A\xrightarrow{\ell}A\).
The question asks whether the analogue of Artin vanishing holds for \(A_{\infty}\)
\cite[Remark~2.11]{BhattSchnellScholze2018}. The idea that Artin vanishing can hold along a tower of
covers of a projective variety is due to Scholze, who proved it in characteristic zero for the tower of
``mock Frobenius'' covers of projective space\footnote{The ``mock Frobenius'' is the toric Frobenius lift
\(\Phi\colon\bbP^{N}\to\bbP^{N}\), \([x_0:\cdots:x_N]\mapsto[x_0^{\ell}:\cdots:x_N^{\ell}]\).} by perfectoid
methods \cite[Theorem~17.3]{Sch14}. The following is
known about Question~\ref{question:BSS} and other vanishing results of a similar nature.

\begin{enumerate}[(a)]

\item In characteristic zero the question has a positive answer. Over a complete algebraically closed
extension of \(\Ql\), it follows by an adaptation of Scholze's perfectoid method, through the perfectoid
cover of \cite[Theorem~1]{BGCHSWY22} and the perfectoid Artin vanishing of
\cite[Theorem~3.3 and Example~3.5]{Rei19}. Over \(\bbC\), Bhatt
\cite[Theorem~4.1]{Bha26} deduces it, for every compact complex torus, from the results of
\cite[\S2]{BhattSchnellScholze2018} on the Fourier--Mellin transform. The vanishing over either field
implies that Question~\ref{question:BSS} has a positive answer over every algebraically closed field of
characteristic zero.

\item For the closely parallel tower of mock Frobenius covers of projective space, Scholze proved the
resulting vanishing for constant \(\Fl\)-sheaves on closed subvarieties in characteristic zero by
perfectoid methods \cite[Theorem~17.3]{Sch14}. Esnault proved the statement for constructible
\(\Fl\)-sheaves over every algebraically closed field of characteristic different from \(\ell\)
\cite[Theorem~5.1]{Esn21}, while Reinecke proved the characteristic zero statement for constructible
\(\Fl\)-sheaves by perfectoid methods \cite[Theorem~3.3 and Example~3.4]{Rei19}.

Bhatt \cite{Bha26} gave an alternative proof of the mock Frobenius
case with the aim of adapting it to the abelian setting. However, his proof relies on the
ramification of \(\Phi\) and it is unclear how to adapt the strategy to the unramified map
\([\ell]\colon A\to A\) \cite[Remark~0.6]{Bha26}.

\item Mousa \cite[Theorem~9.11 and Corollary~9.12]{Mou24} extended the vanishing theorem of \cite{Esn21} to
every tower admitting a levelwise quasi-finite separated morphism to a product of curve towers with cyclic
\(\ell\)-power Galois groups, and constructed the Fourier--Mellin transform for general pro-\(\ell\) Galois
towers \cite[Chapter~10]{Mou24}. In particular, his theorem gives a
positive answer to Question~\ref{question:BSS} when \(A\) is isogenous to a product of elliptic curves
(see also Remark~\ref{rem:BBDG-comparison}).

\end{enumerate}

In positive characteristic Question~\ref{question:BSS} has remained open in general. This article proves
the vanishing in all characteristics. Our methods also give codimension estimates for the Fourier--Mellin
transform (Theorem~\ref{thm:C}) and, as a corollary, generic vanishing for every perverse sheaf on an
abelian variety (Corollary~\ref{cor:D3}).

In the remainder of this section we state the main results of this article.

\subsection{Betti numbers along the tower}
\label{sec:statements}

We begin with an estimate for the growth of the Betti numbers along the tower, which is the main step in
the answer to Question~\ref{question:BSS}.

\begin{thm}\label{thm:A}
Let \(A\) and \(\ell\) be as above and let \(P\in\Perv(A,\Fl)\).\footnote{All perverse sheaves in
this article are for the middle perversity.} For every positive integer \(m\) invertible in \(k\)
and every \(i\in\bbZ\),
\[
  \dim_{\Fl}\rH^{i}\bigl(A,[m]^{*}P\bigr)\ \leq\ C_{P}\,m^{2(g-|i|)},
\]
where \(C_{P}\) is a constant independent of \(m\).\footnote{These bounds are optimal. For
\(A=E_1\times\dots\times E_g\) a product of elliptic curves and \(P=i_{B*}\Fl[b]\) with
\(B=E_1\times\dots\times E_b\subseteq A\) and
\(0\leq b\leq g\), one has \(\dim_{\Fl}\rH^{\pm b}\bigl(A,[m]^{*}P\bigr)=m^{2(g-b)}\).}
\end{thm}

Our proof of Theorem~\ref{thm:A} proceeds by adapting Katz's argument \cite[Theorem~2]{Kat01} to perverse
coefficients. This involves reducing bounds for sums of Betti numbers (and hence for individual Betti
numbers using a Lefschetz argument) to estimates for Euler characteristics. The latter are obtained from
T.~Saito's theory of characteristic cycles and characteristic classes \cite{Sai17}.

Theorem~\ref{thm:A} follows from Theorem~\ref{thm:polarized}, which gives a similar bound for pullback by
any \textit{polarized} \'etale self-map \(f\) of a smooth projective variety (this hypothesis is
restrictive, see Remark~\ref{rem:polarized-endomorphisms}).\footnote{An \'etale endomorphism of a projective
variety is finite, being proper and quasi-finite.}

\subsection{Vanishing along the tower}
\label{sec:tower-statement}

The passage from Theorem~\ref{thm:A} to the vanishing along the tower goes through the Fourier--Mellin
transform, more specifically its completed stalk at the trivial character.

Let \(S=\Fl[[\Tl A]]\) be the completed group algebra of the Tate module, a regular local ring of
dimension \(2g\), and let \(\widehat{\FM}_A(P)\) be the completed stalk at the trivial character of the
\(\Fl\)-linear Fourier--Mellin transform of \(P\), a perfect complex over \(S\)
(Lemma~\ref{lem:FM-finiteness}, compare \cite[\S2]{BhattSchnellScholze2018} and \cite[\S4]{Bha26} for
compact complex tori).

For \(m=\ell^{n}\) let \(S_m=\Fl[A[m]]\), where as usual \(A[m]\) is the group of \(m\)-torsion points of
\(A\). The finite levels of the tower are the derived restrictions of the transform to the quotients \(S_m\),
\begin{equation}\label{eq:finite-level}
  \widehat{\FM}_A(P)\otimes^{L}_{S} S_m\ \simeq\ R\Gamma\bigl(A,[m]^{*}P\bigr).
\end{equation}
Thus Theorem~\ref{thm:A} bounds the cohomology of the transform on infinitesimal neighbourhoods of the
trivial character. These bounds imply that \(\widehat{\FM}_A(P)\in D^{\geq0}(S)\) (see \S\ref{sec:outline}),
and this in turn implies the following theorem.

\begin{thm}\label{thm:B}
Let \(A\) be an abelian variety over an algebraically closed field \(k\), let \(\ell\) be a prime
invertible in \(k\), and let \(F\) be a constructible \(\Fl\)-sheaf on \(A\). There is an integer
\(e\geq0\), depending on \(F\), such that the pullback map
\[
  [\ell^{e}]^{*}\colon\rH^{i}\bigl(A,[\ell^{n}]^{*}F\bigr)\longrightarrow\rH^{i}\bigl(A,[\ell^{n+e}]^{*}F\bigr)
\]
is zero for all \(n\geq0\) and all \(i>\dim\Supp F\). In particular
\[
  \varinjlim_n\ \rH^{i}\bigl(A,[\ell^n]^{*}F\bigr)\ =\ 0
  \qquad\text{for all } i>\dim\Supp F.
\]
\end{thm}

The uniform form of Theorem~\ref{thm:B} arose from a question of Esnault on the defect of Hard Lefschetz
with finite coefficients.

The proof of Lemma~\ref{lem:tower-vanishing} allows us to choose any \(e\) such
that \(\ell^{e}\geq2g+c+1\).
Here \(c\) is an Artin--Rees constant of a free resolution of the completed transform of the Verdier dual
of \(F[\dim\Supp F]\).
Esnault \cite[\S5, Remark~3]{Esn21} raises the question of the least such exponent \(e\) for the mock
Frobenius covers and gives the example of a smooth conic tangent to the three coordinate lines for
\(\ell=2\), for which \(e=2\).
The analogue here is a smooth curve \(C\subseteq A\) whose inclusion factors
through \([\ell^{v}]\) and not through \([\ell^{v+1}]\), for which the least exponent is \(v+1\).

\subsection{Codimension estimates}
\label{sec:codim-statements}

The cohomology sheaves of \(\widehat{\FM}_A(P)\) satisfy the following codimension estimate.

\begin{thm}\label{thm:C}
Let \(P\in\Perv(A,\Fl)\). Then
\[
  \codim_{\Spec S}\ \Supp\ \rH^{i}\bigl(\widehat{\FM}_A(P)\bigr)\ \geq\ 2i \qquad\text{for all } i\geq 0.
\]
\end{thm}

Let \(E/\Ql\) be a finite extension with ring of integers \(\cO\), and let
\[
  R=\cO[[\Tl A]],\qquad \mathfrak R=R\otimes_{\cO}\overline{\bbQ}_{\ell}.
\]
The ring \(\mathfrak R\) is noetherian and Jacobson. Following Esnault--Kerz \cite[\S3.1]{EK21}, we
identify its maximal spectrum \(\operatorname{Spm}(\mathfrak R)\) with the group of continuous
homomorphisms \(\Tl A\to\overline{\bbQ}_{\ell}^{\times}\), endowed with the Zariski topology. For
\(F\in\Dbc(A,\overline{\bbQ}_{\ell})\), let \(\mathfrak{FM}_A(F)\in D^b_{\mathrm{coh}}(\mathfrak R)\) denote
the Fourier--Mellin transform \(\mathfrak{FM}_{\Tl A}(A,F)\) of \cite[Definition~4.4]{EK21}. Theorem~\ref{thm:C}
implies the following codimension estimate.

\begin{cor}\label{cor:D1}
Let \(P\in\Perv(A,\overline{\bbQ}_{\ell})\). Then
\[
  \codim_{\Spec\mathfrak R}\ \Supp\ \rH^{i}\bigl(\mathfrak{FM}_A(P)\bigr)\ \geq\ 2i
  \qquad\text{for all } i\geq 0.
\]
The same estimate holds on \(\operatorname{Spm}(\mathfrak R)\).
\end{cor}

For \(\chi\in\operatorname{Spm}(\mathfrak R)\), let \(L_{\chi}\) be the corresponding rank one
\(\overline{\bbQ}_{\ell}\)-local system on \(A\). Corollary~\ref{cor:D1} and Verdier duality give
\[
  \codim_{\operatorname{Spm}(\mathfrak R)}
  \bigl\{\chi\mid\rH^i(A,P\otimes L_{\chi})\neq0\bigr\}\ \geq\ 2|i|
  \qquad(i\in\bbZ)
\]
for the jump loci (Corollary~\ref{cor:jumping-loci}). Together with
Corollary~\ref{cor:connectivity}(2), this gives the following generic vanishing statements.

\begin{cor}\label{cor:D3}
Let \(P\in\Perv(A,\Lambda)\).
\begin{enumerate}[(1)]
\item If \(\Lambda=\bbF_{\ell}\), there is a closed subset \(Z\subseteq\Spec S\) of codimension at least
\(2\) such that, on \(\Spec S\setminus Z\), the complex \(\widehat{\FM}_A(P)\) is a locally free sheaf
concentrated in degree \(0\), whose rank is \(\chi(A,P)\geq0\).
\item If \(\Lambda=\overline{\bbQ}_{\ell}\), there is a closed subset
\(Z\subseteq\operatorname{Spm}(\mathfrak R)\) of codimension at least \(2\) such that, for every
\(\chi\notin Z\),
\[
  \rH^{i}(A,P\otimes L_{\chi})=0\quad(i\neq0),\qquad
  \dim_{\overline{\bbQ}_{\ell}}\rH^{0}(A,P\otimes L_{\chi})=\chi(A,P)\geq0.\footnotemark
\]
\end{enumerate}
\footnotetext{For non-negativity of \(\chi(A,P)\), see Franecki--Kapranov
\cite[Corollary~1.4]{FraneckiKapranov00} over \(\bbC\), and Debarre--Moonen
\cite[Proposition~3.7]{DebarreMoonen25} in arbitrary characteristic.}
\end{cor}

Part (1) is Corollary~\ref{cor:connectivity}(2). We prove (2) in \S\ref{sec:algebraic-model}.

\subsection{Comparison with earlier work}
\label{sec:comparison}

We compare Theorem~\ref{thm:C} and Corollaries~\ref{cor:D1} to~\ref{cor:D3} with earlier work.

\emph{Characteristic zero.} Let \(A\) be a complex abelian variety.
\begin{enumerate}[(a)]
\item \emph{Coefficients of characteristic zero.} Kr\"amer--Weissauer
\cite[Theorem~1.1]{KraemerWeissauer2015} and Schnell \cite[Corollary~7.5]{Schnell2016} proved generic
vanishing for perverse sheaves with \(\bbC\)-coefficients, and Schnell proved the bound \(2|i|\) for the
jump loci with \(\bbC\)-coefficients \cite[Theorem~7.4]{Schnell2016}. Bhatt--Schnell--Scholze proved the
same bound with any characteristic zero coefficient field \cite[Theorem~3.1]{BhattSchnellScholze2018}.
Their proof uses Hard Lefschetz for rank one twists of arbitrary simple perverse sheaves.\footnote{Hard Lefschetz for every semisimple perverse sheaf, and more generally for every semisimple
holonomic \(\mathcal D\)-module, was conjectured by Kashiwara \cite{Kas98}. For perverse sheaves, Drinfeld
\cite{Dri01} deduced it from de Jong's conjecture on representations of arithmetic fundamental groups over
finite fields \cite{deJong01}, using Lafforgue's theorem \cite{Laf02}. Gaitsgory \cite{Gai07} proved the
required case of de Jong's conjecture, and B\"ockle--Khare \cite{BoeckleKhare06} proved it
independently under additional hypotheses. Sabbah \cite{Sab05} proved the case of semisimple local systems
through polarizable twistor \(\mathcal D\)-modules, and Mochizuki \cite{Moc11} the general case.}

\item \emph{General coefficients.} Bhatt--Schnell--Scholze proved generic vanishing for an arbitrary
coefficient field on a compact complex torus, using Artin vanishing on the universal cover
\cite[Corollary~2.10]{BhattSchnellScholze2018}. Their argument gives codimension at least \(|i|\) for the
jump loci in degree \(i\) \cite[\S1.2]{BhattSchnellScholze2018}.
\end{enumerate}

For perverse sheaves on a complex abelian variety with coefficients in a finite field of characteristic
\(\ell\), our argument improves the bound \(|i|\) to \(2|i|\) for the jump loci on the whole character variety
(Remark~\ref{cor:betti-torus}).

\emph{Characteristic \(p>0\).} Let \(A\) be an abelian variety over an algebraically closed field of
characteristic \(p\ne\ell\).
\begin{enumerate}[(a)]
\item \emph{Generic vanishing.} Weissauer proved generic vanishing when \(A\) and the
\(\overline{\bbQ}_{\ell}\)-perverse sheaf are defined over a finite field, and also for perverse sheaves
of geometric origin \cite[Introduction]{Wei16}.

\item \emph{Codimension estimates.} Esnault--Kerz proved the bound \(2|i|\) for the jump loci of
arithmetic \(\overline{\bbQ}_{\ell}\)-perverse sheaves \cite[Definition~5.1 and Theorem~1.5]{EK21}.
Their proof uses Hard Lefschetz for arbitrary rank one twists of arithmetic semisimple perverse sheaves
\cite[Theorem~5.4]{EK21}.
\end{enumerate}

Corollary~\ref{cor:D1} removes the arithmeticity hypothesis and hence gives the bound \(2|i|\) for the
jump loci of every \(\ell\)-adic perverse sheaf. Generic vanishing likewise holds for every perverse sheaf by
Corollary~\ref{cor:D3}.

\subsection{Outline of the proofs}
\label{sec:outline}

Figure~\ref{fig:leitfaden} summarizes the principal steps of the proofs.

\begin{figure}[t]
\centering
\begin{tikzpicture}[
  >=stealth,
  every node/.style={align=center},
  result/.style={draw, rounded corners, thick, text width=6.1cm,
    minimum height=1cm, inner sep=5pt, font=\small},
  source/.style={font=\footnotesize\itshape, text=black!65},
  dependency/.style={->, thick},
  explanation/.style={font=\footnotesize, fill=white, inner sep=2pt}
]
\node[source] (saito) at (-3.8,9.3) {Characteristic classes\\and transversality};
\node[source] (artin) at (3.8,9.3) {Artin vanishing\\and purity};
\node[result] (euler) at (-3.8,7.7)
  {Euler characteristic bounds\\Proposition~\ref{prop:graded-estimate}};
\node[result] (recursion) at (3.8,7.7)
  {Betti number recursion\\Lemmas~\ref{lem:one-cut} and~\ref{prop:recursion}};
\node[result] (betti) at (0,5.8)
  {Betti number bounds\\Theorems~\ref{thm:polarized} and~\ref{thm:A}};
\node[result] (codimension) at (-3.8,1.85)
  {Codimension estimates over \(\Fl\)\\Theorem~\ref{thm:C}};
\node[result] (connectivity) at (3.8,1.85)
  {Completed transform\\\(\widehat{\FM}_A(P)\in D^{\geq0}(S)\)};
\node[result] (elladic) at (-3.8,-0.65)
  {\(\ell\)-adic codimension estimates\\Corollary~\ref{cor:D1}};
\node[result] (tower) at (3.8,-0.65)
  {Vanishing along the tower\\Theorem~\ref{thm:B}};
\draw[dependency] (saito) -- (euler);
\draw[dependency] (artin) -- (recursion);
\draw[dependency] (euler) -- (betti);
\draw[dependency] (recursion) -- (betti);
\draw[thick] (betti.south) -- (0,4.0)
  node[explanation, midway] {\(m=\ell^n\)\\\ref{item:M1} and duality};
\draw[dependency] (0,4.0) -- (-3.8,4.0) -- (codimension.north)
  node[explanation, midway] {Proposition~\ref{prop:HS-machine}\\applied to \(\widehat{\FM}_A(P)\)};
\draw[dependency] (0,4.0) -- (3.8,4.0) -- (connectivity.north)
  node[explanation, midway] {Proposition~\ref{prop:HS-machine}\\applied to \(\RHom_S(\widehat{\FM}_A(P),S)\), and \ref{item:M3}};
\draw[dependency] (codimension) -- (elladic)
  node[explanation, midway] {integral forms and specialization\\
    Lemmas~\ref{lem:special-generic-codim} and~\ref{lem:perverse-lattice}};
\draw[dependency] (connectivity) -- (tower)
  node[explanation, midway] {Lemma~\ref{lem:tower-vanishing}\\and perverse truncation};
\end{tikzpicture}
\caption{Leitfaden}
\label{fig:leitfaden}
\end{figure}
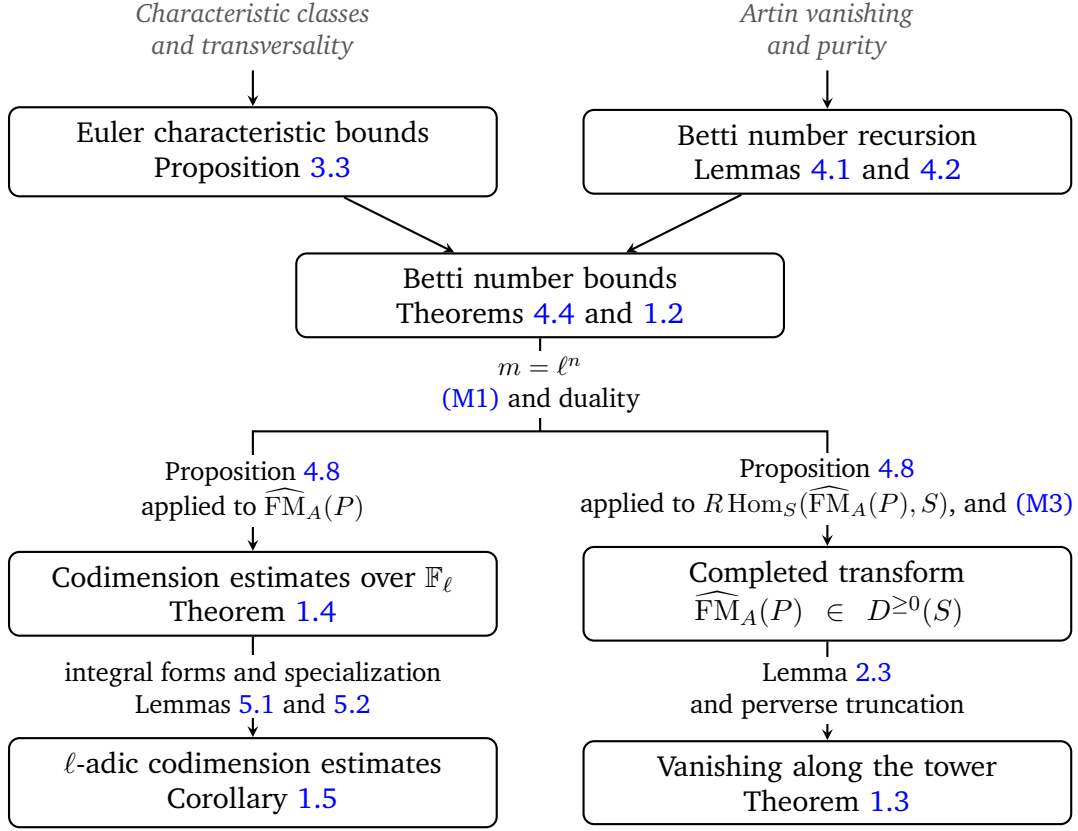

Our use of Betti number estimates to prove Theorem~\ref{thm:C} was motivated by the second
author's recent joint work with Bah on uniform generic vanishing for tori
\cite{BahShuddhodan26}. There the complexity estimates of \cite{SFFK23} give bounds for the
number of points of the cohomology jump loci, and the exponents in these bounds measure codimension
(see Corollary~\ref{cor:character-count}(2) for a similar result).\footnote{This avoids making
effective the description of these loci as finite unions of translates of algebraic cotori, proved by
Gabber--Loeser under resolution of singularities and principalization hypotheses
\cite[Th\'eor\`eme~6.3.3]{GabberLoeser1996}.}

We first explain how Euler characteristic bounds on complete intersection flags give the Betti number
estimates of Theorem~\ref{thm:A}. Fix a very ample divisor \(H\) on \(A\), let \(m\geq1\) be invertible in
\(k\), and let \(P_m=[m]^{*}P\).
Lemma~\ref{lem:flags} gives a flag \(Y_0\subset\dots\subset Y_g=A\) of smooth complete intersections
with \(\dim Y_r=r\). For \(Q_r=P_m|^*_{Y_r}[r-g]\), each inclusion
\(Y_{r-1}\hookrightarrow Y_r\) is properly \(\SSupp(Q_r)\)-transversal, as in \ref{item:S5}.
Successive divisor pullback and Saito's index formula express \(\chi(Y_r,Q_r)\) through the
characteristic numbers \(c_a(P_m)\) by Proposition~\ref{prop:graded-estimate}. Lemma~\ref{lem:scaling}
gives \(c_a(P_m)=m^{2(g-a)}c_a(P)\), so there is a constant \(C_1\), depending only on \(P\) and
\(H\), such that \(\lvert\chi(Y_r,Q_r)\rvert\leq C_1m^{2r}\). Artin vanishing on the
successive affine complements gives the estimate of Lemma~\ref{lem:one-cut} for a single ample divisor,
and Lemma~\ref{prop:recursion} gives the recursion
\[
  B_{r}\ \leq\ \bigl|\chi(Y_{r},Q_{r})\bigr|+4B_{r-1},
  \qquad
  B_{r}=\sum_a\dim\rH^a(Y_{r},Q_{r}).
\]
Different cohomological degrees are controlled by members of the flag of different dimensions.
For \(0\leq i\leq g\), iterating Lemma~\ref{lem:one-cut} through \(i\) cuts realizes \(\rH^i(A,P_m)\) as a
subquotient of \(\rH^{0}(Y_{g-i},Q_{g-i})\), where \(Y_{g-i}\) has dimension \(g-i\). Consequently,
for a constant \(C_2\) depending only on \(P\) and \(H\),
\[
  \dim_{\Fl}\rH^i(A,P_m)\ \leq\ B_{g-i}\ \leq\ C_2m^{2(g-i)} \qquad (0\leq i\leq g).
\]
Verdier duality gives the negative degrees. This proves Theorem~\ref{thm:A}, which is deduced from
Theorem~\ref{thm:polarized} in \S\ref{sec:recursion}.

The remaining steps use commutative algebra over the regular local ring \(S\). For \(m=\ell^n\), the
isomorphism \eqref{eq:finite-level}, including its trace transition maps, is \ref{item:M1} at finite level.
For a finite \(S\)-module \(N\), the length of \(N\otimes_S S_m\) grows like \(m^{\dim\Supp N}\), so length
bounds along the tower constrain supports. Theorem~\ref{thm:A} gives the length bounds required by
Proposition~\ref{prop:HS-machine}, which converts them into codimension estimates using the hyper-Tor
spectral sequence and bounds for Tor groups. Applied to \(\widehat{\FM}_A(P)\), the proposition proves
Theorem~\ref{thm:C}. Applied to \(\RHom_S(\widehat{\FM}_A(P),S)\), whose derived restriction to \(S_m\)
is the \(S_m\)-linear dual of \(R\Gamma(A,[m]^{*}P)\), the proposition gives
\[
  \codim\Supp\rH^{j}\bigl(\RHom_S(\widehat{\FM}_A(P),S)\bigr)\ \geq\ 2j.
\]
The criterion \ref{item:M3} then gives \(\widehat{\FM}_A(P)\in D^{\geq0}(S)\) for every perverse \(P\), in
particular for \(\bbD P\). By Lemma~\ref{lem:tower-vanishing}, \(\widehat{\FM}_A(\bbD P)\) is the
\(\Fl\)-linear dual of the direct limit of the complexes \(R\Gamma(A,[\ell^n]^*P)\), which therefore lies in
\(D^{\leq0}\), and the Artin--Rees lemma, applied to a free resolution of \(\widehat{\FM}_A(\bbD P)\), shows
that the pullback maps of the tower vanish in positive degrees after a fixed number of steps. For a
constructible sheaf \(F\) with \(s=\dim\Supp F\), the shift \(F[s]\) lies in \({}^pD^{\leq0}\), and the
perverse cohomology filtration of \(\bbD(F[s])\) together with the exactness of \(\widehat{\FM}_A\) reduces
Theorem~\ref{thm:B} to this case. Thus Theorems~\ref{thm:B} and~\ref{thm:C} are consequences of
Theorem~\ref{thm:A}.

The article is organized as follows. \S\ref{sec:preliminaries} constructs the completed stalk at the
trivial character of the \(\Fl\)-linear Fourier--Mellin transform and recalls the results on
characteristic cycles and transversality used later. \S\ref{sec:chi} proves the
Euler characteristic bound for the complete intersection flags. \S\ref{sec:proofs} proves
Theorems~\ref{thm:A}, \ref{thm:B} and~\ref{thm:C} and deduces Corollary~\ref{cor:connectivity}.
In \S\ref{sec:ell-adic}, we extend the Betti number and codimension estimates to \(\ell\)-adic
coefficients and deduce bounds for torsion characters and generic vanishing.
Remark~\ref{cor:betti-torus} gives the corresponding estimates on the Betti character variety of a
complex abelian variety for finite coefficient fields.
The short Appendix~\ref{sec:integral-perverse} shows that every perverse sheaf with coefficients in a finite
extension of \(\Ql\) has a torsion-free integral form, in the terminology of Juteau \cite{Jut09}, and that
its modular reduction is perverse. These integral forms are used in \S\ref{sec:ell-adic}.

\subsection*{Acknowledgements}

We thank H\'el\`ene Esnault for her encouragement, for her comments on the article, and for her questions
on the defect of Hard Lefschetz with finite and integral coefficients, which led to the uniform form of
Theorem~\ref{thm:B} and to Remark~\ref{rem:integral-uniform}. We thank Thomas Kr\"amer for the references
in Remark~\ref{rem:katz-cc}.

\section{Preliminaries}
\label{sec:preliminaries}

We begin by fixing conventions which shall be used in the rest of the article.

\subsection{Notation and conventions}
\label{sec:conventions}

\begin{para}\label{para:conventions}
Throughout, \(k\) is an algebraically closed field and \(\ell\) is a prime invertible in \(k\). Let
\(\Lambda\) denote the coefficient ring, which unless otherwise specified is \(\Fl\). Since \(k\) is
algebraically closed, we suppress Tate twists.

For a scheme \(X\) of finite type over \(k\), let \(\Lambda_X\), or simply \(\Lambda\) when \(X\) is clear
from the context, denote the constant sheaf with value \(\Lambda\). Let \(\Dbc(X,\Lambda)\) denote the
bounded derived category of \(\Lambda\)-sheaves with constructible cohomology on the \'etale site of
\(X\), and let \(\Perv(X,\Lambda)\subseteq\Dbc(X,\Lambda)\) be the heart of the middle perverse
\(t\)-structure \cite{BBDG18}. Unless otherwise stated, cohomology of schemes of finite type over \(k\)
is \'etale cohomology.

Let \(\CH_a(X)\) denote the group of \(a\)-dimensional cycles on \(X\) with integer coefficients
modulo rational equivalence, and let \(\CH_*(X)=\bigoplus_a\CH_a(X)\).

Let \(A\) be an abelian variety over \(k\) of dimension \(g\geq1\). Unless stated otherwise, \(m\) denotes a
power \(\ell^n\) of \(\ell\), with \(n\geq0\). The multiplication map \([m]\colon A\to A\) is an \'etale
isogeny of degree \(m^{2g}\). Verdier duality on \(A\) is normalized by
\(\bbD(-)=R\mathcal{H}om_A(-,\Lambda_A[2g])\), so that \(\bbD\) preserves \(\Perv(A,\Lambda)\). Since \([m]\)
is \'etale, \([m]^*=[m]^!\) is perverse \(t\)-exact and commutes with \(\bbD\).
\end{para}

\begin{para}\label{para:transitions}
For \(F\in\Dbc(A,\Lambda)\) and powers \(m\mid m'\) of \(\ell\), pullback by \([m'/m]\) gives
\[
  \rH^{i}(A,[m]^{*}F)\longrightarrow\rH^{i}(A,[m']^{*}F).
\]
The direct limit of this system is the one in Question~\ref{question:BSS}.

Dually, the trace maps of the \'etale \([m'/m]\) induce
\(\rH^{i}(A,[m']^{*}F)\to\rH^{i}(A,[m]^{*}F)\). Verdier duality exchanges the two systems. Since \(A\)
is proper and \([m]^{*}\bbD F=\bbD[m]^{*}F\), there are natural isomorphisms
\[
  R\Gamma\bigl(A,[m]^{*}F\bigr)\ \simeq\ \RHom_{\Fl}\bigl(R\Gamma(A,[m]^{*}\bbD F),\Fl\bigr),
  \qquad
  \rH^{i}(A,[m]^{*}F)^{\vee}\ \cong\ \rH^{-i}(A,[m]^{*}\bbD F),
\]
compatible with transitions, since duality carries the unit \(\mathrm{id}\to[m'/m]_{*}[m'/m]^{*}\) to
the counit \([m'/m]_{*}[m'/m]^{!}\to\mathrm{id}\).
\end{para}

\subsection{The Fourier--Mellin transform completed at the trivial character}
\label{sec:mellin}

\begin{para}\label{para:group-algebra}
Following \cite[\S3.1]{EK21}, for a profinite abelian group \(G\), define the completed group algebra (with \(\Fl\)-coefficients) by
\[
  \Fl[[G]]\ :=\ \varprojlim_U \Fl[G/U],
\]
where \(U\) ranges over the open subgroups of \(G\). Let
\(\Tl A=\varprojlim_n A[\ell^n]\cong\Zl^{2g}\) be the \(\ell\)-adic Tate module. The subgroups
\(\ell^n\Tl A\) are cofinal among the open subgroups of \(\Tl A\), and the projection to
\(A[\ell^n]\) identifies \(\Tl A/\ell^n\Tl A\) with \(A[\ell^n]\). Thus
\[
  S\ :=\ \Fl[[\Tl A]]\ \simeq\ \varprojlim_n\ \Fl\bigl[A[\ell^n]\bigr].
\]
Choose a \(\Zl\)-basis \(\gamma_1,\dots,\gamma_{2g}\) of \(\Tl A\), and let \(\gamma_{j,m}\) denote
the image of \(\gamma_j\) in \(A[m]\). Reducing this basis modulo \(m\) gives an isomorphism
\((\bbZ/m)^{2g}\xrightarrow{\sim}A[m]\) sending the \(j\)-th standard generator to \(\gamma_{j,m}\).
For every \(m=\ell^n\), there is therefore an isomorphism
\[
  \phi_m\colon
  \Fl[x_1,\dots,x_{2g}]/(x_1^m,\dots,x_{2g}^m)
  \ \xrightarrow{\ \sim\ }\ S_m:=\Fl[A[m]],
  \qquad x_j\longmapsto[\gamma_{j,m}]-1.
\]
Indeed \([\gamma_{j,m}]^m=1\), and
\((1+x_j)^m-1=x_j^m\) because \(m\) is a power of \(\ell\). For \(m\mid m'\), let
\(q_{m',m}\) be the quotient on the left sending each \(x_j\) to \(x_j\), and let
\(\tau_{m',m}\colon S_{m'}\to S_m\) be induced by \([m'/m]\colon A[m']\to A[m]\). Since
\([m'/m](\gamma_{j,m'})=\gamma_{j,m}\), we have
\[
  \tau_{m',m}\circ\phi_{m'}\ =\ \phi_m\circ q_{m',m}.
\]
Thus the isomorphisms \(\phi_m\) are compatible with the transition maps.

Let \(\frm=(x_1,\dots,x_{2g})\) and \(I_m=(x_1^m,\dots,x_{2g}^m)\). The inclusions
\[
  \frm^{2g(m-1)+1}\ \subseteq\ I_m\ \subseteq\ \frm^m
\]
show that the ideals \(I_m\) are cofinal with the powers of \(\frm\). Passing to inverse limits in the
compatible isomorphisms above gives a topological isomorphism
\[
  \Fl[[x_1,\dots,x_{2g}]]\ \xrightarrow{\ \sim\ }\ S,
  \qquad 1+x_j\longmapsto[\gamma_j],
\]
under which \(S\to S_m\) is the quotient by \(I_m\). In particular, \(S\) is a complete regular local
ring of dimension \(2g\) with maximal ideal \(\frm\) and residue field \(\Fl\).
\end{para}

\begin{para}\label{para:tautological}
We take the identity as base point and let \(\pi_1(A)=\pi_1(A,0)\). For \(m=\ell^{n}\), the
multiplication map \([m]\colon A\to A\) is an \(A[m]\)-torsor for the translation action, hence a
connected Galois cover classified by a surjection \(\pi_1(A)\twoheadrightarrow A[m]\)
\cite[Tag~\href{https://stacks.math.columbia.edu/tag/03SF}{03SF}]{stacks-project}.
These surjections are compatible with the maps \([m'/m]\colon A[m']\to A[m]\) and induce a surjection
\(\pi_1(A)\to\Tl A\).
For \(v\in\Tl A\) with images \(v_m\in A[m]\), the basis elements \([v_m]\in S_m\) are compatible under
the maps \(\tau_{m',m}\) of (\ref{para:group-algebra}) and define a unit \([v]\in S\) with inverse
\([-v]\). The continuous tautological character is
\[
  \rho\colon\pi_1(A,0)\twoheadrightarrow\Tl A\xrightarrow{\ v\mapsto[v]\ }S^{\times},
\]
the canonical character of \cite[\S4.1]{EK21} for \(\pi=\Tl A\).

For every Artinian quotient \(S'\) of \(S\), let \(L_{S'}\) be the rank one \(S'\)-local system
attached to the character induced by \(\rho\). A further quotient \(S'\twoheadrightarrow S''\)
induces a transition map \(L_{S'}\to L_{S''}\) and an isomorphism
\(L_{S'}\otimes_{S'}S''\simeq L_{S''}\). Let \(\Ldot=(L_{S'})_{S'}\) be this projective system.
Its restriction to \(S/\frm^{N}\), for \(N\geq1\), is an \(\frm\)-adic projective system with lisse
terms \cite[Expos\'e~V, D\'efinition~3.1.1]{SGA5}, and defines an object of the heart of the
\(t\)-structure on adic complexes of \cite[Theorem~6.3(i)]{Eke90}.

For \(F\in\Dbc(A,\Lambda)\), let \(F\otimes_{\Lambda}\Ldot\) denote the constructible adic complex
\((F\otimes_{\Lambda}L_{S/\frm^{N}})_{N\geq1}\). Following \cite[\S4.1]{EK21}, define
\[
  \widehat{\FM}_A(F)\ :=\
  R\varprojlim_{N\geq1}R\Gamma\bigl(A,F\otimes_{\Lambda}L_{S/\frm^{N}}\bigr)\ \in\ D(S).
\]
This is the completed stalk at the trivial character of the \(\Fl\)-linear Fourier--Mellin transform
of \(F\).

The first lemma below identifies the complexes
\(R\Gamma(A,F\otimes_{\Lambda}L_{S_m})\) with the tower of Question~\ref{question:BSS}, and the second
proves the finiteness of \(\widehat{\FM}_A(F)\) and its base change to the Artinian quotients of \(S\).
\end{para}

\begin{lem}\label{lem:finite-level}
For every \(F\in\Dbc(A,\Lambda)\) there are isomorphisms
\[
  R\Gamma\bigl(A,F\otimes_{\Lambda}L_{S_m}\bigr)\ \simeq\ R\Gamma\bigl(A,[m]^{*}F\bigr)
\]
compatible with the transition maps, those induced by \(\Ldot\) on the left and the trace maps of
(\ref{para:transitions}) on the right.
\end{lem}

\begin{proof}
By (\ref{para:tautological}), \(L_{S_m}\) is the local system attached to the representation of
\(\pi_1(A)\) on \(S_m=\Lambda[A[m]]\) through the surjection \(\pi_1(A)\twoheadrightarrow A[m]\) and
the regular representation of \(A[m]\). Since \([m]\) is finite Galois with group \(A[m]\), acting by
translation on the source, \([m]_{*}\Lambda_A\) is the local system attached to the same representation
\cite[Tags~\href{https://stacks.math.columbia.edu/tag/03RV}{03RV}
and~\href{https://stacks.math.columbia.edu/tag/0DV5}{0DV5}]{stacks-project}, so that
\([m]_{*}\Lambda_A\simeq L_{S_m}\) as \(S_m\)-local systems. Moreover this is compatible with
the transition maps, the quotient \(S_{m'}\to S_m\) corresponding to \([m]_{*}\) of the trace
\([m'/m]_{*}\Lambda_A\to\Lambda_A\), both being induced by the quotient \(A[m']\to A[m]\).

The projection formula for the finite map \([m]\) then gives
\[
  F\otimes_{\Lambda}L_{S_m}
  \ \simeq\ F\otimes_{\Lambda}[m]_{*}\Lambda_A
  \ \simeq\ [m]_{*}\bigl([m]^{*}F\bigr),
\]
and applying \(R\Gamma(A,-)\), with \(R\Gamma(A,[m]_{*}Q)\simeq R\Gamma(A,Q)\) for the finite morphism
\([m]\), gives the isomorphism. The transition for \(m\mid m'\) is induced by the trace of
\([m'/m]\), which is the trace map of (\ref{para:transitions}).
\end{proof}

\begin{lem}\label{lem:FM-finiteness}
Let \(F\in\Dbc(A,\Lambda)\). The complex \(\widehat{\FM}_A(F)\) is perfect over \(S\), and for every
Artinian quotient \(S'\) of \(S\) there is a natural isomorphism
\[
  \widehat{\FM}_A(F)\otimes^{L}_{S}S'\ \simeq\ R\Gamma\bigl(A,F\otimes_{\Lambda}L_{S'}\bigr),
\]
compatible with the transition maps of \(\Ldot\) of (\ref{para:tautological}) as \(S'\) varies.%
\footnote{See also \cite[Proposition~4.1(1)--(2)]{EK21} for the same properties with integral
coefficients.}
\end{lem}

\begin{proof}
Since \(S\) is regular, every complex in \(D^{b}_{\mathrm{coh}}(S)\) is perfect. By
\cite[Theorem~6.3(iii)]{Eke90}, the direct image \(R\Gamma(A,F\otimes_{\Lambda}\Ldot)\) is a bounded
constructible adic complex over the point, and its reduction modulo \(\frm^{N}\) is
\(R\Gamma(A,F\otimes_{\Lambda}L_{S/\frm^{N}})\). Since \(S\) is complete, the equivalences of
\cite[Theorems~7.1(i) and~7.2(i)]{Eke90} give \(\widehat{\FM}_A(F)\in D^{b}_{\mathrm{coh}}(S)\) and,
compatibly as \(N\) varies, natural isomorphisms
\[
  \widehat{\FM}_A(F)\otimes^{L}_{S}S/\frm^{N}\ \simeq\ R\Gamma(A,F\otimes_{\Lambda}L_{S/\frm^{N}}).
\]

For an Artinian quotient \(S'\), choose \(N\geq1\) with \(\frm^{N} S'=0\) and let \(B=S/\frm^{N}\).
Then \(S'\) is a \(B\)-algebra and \(L_B\otimes^{L}_{B} S'\simeq L_{S'}\). The projection formula for
the proper morphism \(A\to\Spec k\) \cite[Expos\'e~XVII, Proposition~5.2.9]{ArtinGrothendieckSGA4} gives
\[
\begin{aligned}
  R\Gamma(A,F\otimes_{\Lambda}L_B)\otimes^{L}_{B} S'
  &\simeq R\Gamma(A,F\otimes_{\Lambda}L_B\otimes^{L}_{B} S')\\
  &\simeq R\Gamma(A,F\otimes_{\Lambda}L_{S'}).
\end{aligned}
\]
Combining this with the isomorphism for \(B\) proves the statement for \(S'\). Naturality gives
compatibility as \(S'\) varies.
\end{proof}

We collect the consequences needed below.

\begin{enumerate}[label=\textup{(M\arabic*)}, ref=\textup{(M\arabic*)}]

\item\label{item:M1} For every \(F\in\Dbc(A,\Lambda)\), the complex \(\widehat{\FM}_A(F)\) is perfect over
\(S\). For every \(m=\ell^n\), there is a natural isomorphism
\[
  \widehat{\FM}_A(F)\otimes^L_S S_m\ \simeq\ R\Gamma\bigl(A,[m]^{*}F\bigr),
\]
under which, for \(m\mid m'\), the transition map on the right is the trace of \([m'/m]\). This follows
from Lemmas~\ref{lem:finite-level} and~\ref{lem:FM-finiteness}.

\item\label{item:M2} For integers \(a\leq b\), a perfect complex \(C\) over \(S\) with
\(C\otimes^{L}_{S}\Fl\in D^{[a,b]}\) has Tor-amplitude in \([a,b]\). It therefore has a representative
by finite free \(S\)-modules concentrated in these degrees
\cite[Tags~\href{https://stacks.math.columbia.edu/tag/068V}{068V}
and~\href{https://stacks.math.columbia.edu/tag/0658}{0658}]{stacks-project}. Thus
\(C\in D^{[a,b]}(S)\) and \(\RHom_S(C,S)\in D^{[-b,-a]}(S)\).

If \(P\) is perverse, then \(R\Gamma(A,P)\in D^{[-g,g]}\), since \(A\) is proper of dimension \(g\)
\cite[4.2.4]{BBDG18}. By \ref{item:M1} at \(m=1\),
\(\widehat{\FM}_A(P)\) is perfect and \(\widehat{\FM}_A(P)\otimes^{L}_{S}\Fl\simeq R\Gamma(A,P)\). Applying
the preceding argument with \(C=\widehat{\FM}_A(P)\) and \([a,b]=[-g,g]\) shows that \(\widehat{\FM}_A(P)\)
and \(\RHom_S(\widehat{\FM}_A(P),S)\) lie in \(D^{[-g,g]}(S)\).

\item\label{item:M3} The ring \(S\) is regular, hence Gorenstein with dualizing complex \(S[2g]\).
For every perfect complex \(Q\), \cite[Lemma~2.8 and Remark~2.9]{BhattSchnellScholze2018} gives%
\footnote{The support of the zero module has codimension \(+\infty\).}
\[
  \RHom_S(Q,S)\in D^{\geq0}(S)
  \quad\Longleftrightarrow\quad
  \codim\Supp\rH^{i}(Q)\ \geq\ i\quad\text{for every }i.
\]
Equivalently, for every perfect complex \(K\),
\[
  K\in D^{\geq0}(S)
  \quad\Longleftrightarrow\quad
  \codim\Supp\rH^{i}\bigl(\RHom_S(K,S)\bigr)\ \geq\ i
  \quad\text{for every }i.
\]

\end{enumerate}

The following lemma expresses vanishing along the multiplication tower in terms of the completed
Fourier--Mellin transform. We use Verdier duality at each finite level, as in the proof of
\cite[Theorem~4.1]{Bha26} for compact complex tori.%
\footnote{Bhatt uses \cite[Lemma~4.2(3)]{Bha26} to identify the direct limit
\(\varinjlim_{n}R\Gamma(A,[\ell^{n}]^{*}F)\), taken with pullback transition maps, with local
cohomology of the Grothendieck dual of the completed transform. For the vanishing criterion here, it
suffices to compute the \(\Fl\)-linear dual of this direct limit, as in \eqref{eq:colimit-dual}.}

\begin{lem}\label{lem:tower-vanishing}
Let \(F\in\Dbc(A,\Lambda)\) and let \(K=\widehat{\FM}_A(\bbD F)\). There is a natural isomorphism in
\(D(\Fl)\)
\begin{equation}\label{eq:colimit-dual}
  K\ \simeq\ \RHom_{\Fl}\Bigl(\varinjlim_{n}R\Gamma\bigl(A,[\ell^{n}]^{*}F\bigr),\ \Fl\Bigr),
\end{equation}
where the transition maps are the pullback maps. Moreover, the following are equivalent.
\begin{enumerate}
\item \(K\in D^{\geq0}(S)\).
\item \(\varinjlim_n R\Gamma\bigl(A,[\ell^n]^{*}F\bigr)\in D^{\leq0}(\Fl)\).
\item There is an integer \(e\geq0\) such that the pullback map
\[
  [\ell^{e}]^{*}\colon\rH^{i}\bigl(A,[\ell^{n}]^{*}F\bigr)\longrightarrow\rH^{i}\bigl(A,[\ell^{n+e}]^{*}F\bigr)
\]
is zero for all \(n\geq0\) and all \(i>0\).
\end{enumerate}
\end{lem}

\begin{proof}
Let \(C_n=R\Gamma(A,[\ell^n]^{*}F)\) and \(C=\varinjlim_n C_n\). The definition of \(K\), cofinality in
(\ref{para:group-algebra}), Lemma~\ref{lem:finite-level} and the duality of (\ref{para:transitions}) give
\[
  K\ \simeq\ R\varprojlim_n\RHom_{\Fl}(C_n,\Fl),
\]
with the trace maps corresponding to the duals of the pullback maps. For every integer \(a\), the
inverse system \((\rH^{-a}(C_n)^{\vee})_n\) consists of finite-dimensional vector spaces, so it
satisfies the Mittag--Leffler condition and has vanishing \(\varprojlim^{1}\).
Since filtered colimits and \(\Fl\)-linear duality are exact, the natural comparison map
\[
  \RHom_{\Fl}(C,\Fl)\ \longrightarrow\ K
\]
induces, for every \(a\), the isomorphisms
\[
  \bigl(\varinjlim_n\rH^{-a}(C_n)\bigr)^{\vee}
  \ \simeq\ \varprojlim_n\rH^{-a}(C_n)^{\vee}
  \ \simeq\ \rH^a(K).
\]
This proves the isomorphism in \eqref{eq:colimit-dual}.

A vector space vanishes if and only if its dual vanishes, so
\(C\in D^{\leq0}(\Fl)\) if and only if \(K\in D^{\geq0}(\Fl)\). The latter is equivalent to
\(K\in D^{\geq0}(S)\). This proves that (1) and (2) are equivalent.

Since filtered colimits are exact, \(\rH^{i}(C)=\varinjlim_n\rH^{i}(C_n)\), and (3) implies (2).

Finally we assume (1) and prove (3). By \ref{item:M1} at \(m=1\) and \ref{item:M2}, \(K\) is represented by a bounded complex
\(D^{\bullet}\) of finite free \(S\)-modules, which by (1) is exact in every degree \(j<0\). Let
\(B^{j+1}:=d(D^{j})\subseteq D^{j+1}\). Since \(D^{\bullet}\) is bounded, the Artin--Rees lemma
\cite[Tag~\href{https://stacks.math.columbia.edu/tag/00IN}{00IN}]{stacks-project} gives an integer
\(c\geq0\) such that
\[
  B^{j+1}\cap\frm^{r}D^{j+1}\ \subseteq\ \frm^{r-c}B^{j+1}
  \qquad\text{for all } r\geq c \text{ and all } j<0.
\]
Choose \(e\geq0\) with \(\ell^{e}\geq2g+c+1\). We claim
\[
  [\ell^{e}]^{*}\colon\rH^{i}\bigl(A,[\ell^{n}]^{*}F\bigr)\longrightarrow\rH^{i}\bigl(A,[\ell^{n+e}]^{*}F\bigr)
\]
is zero for all \(n\geq0\) and all \(i>0\).

Let \(m=\ell^{n}\) and \(m'=\ell^{n+e}\). For a finite
free \(S\)-module \(D\) one has \(D\otimes_{S}S_m=D/I_mD\), so \(K\otimes^{L}_{S}S_m\) is computed by
\(D^{\bullet}/I_mD^{\bullet}\), and the map \(K\otimes^{L}_{S}S_{m'}\to K\otimes^{L}_{S}S_m\) induced by
the quotient \(S_{m'}\to S_m\) by the projection
\(D^{\bullet}/I_{m'}D^{\bullet}\to D^{\bullet}/I_mD^{\bullet}\).

Let \(j<0\) and let \(z\in D^{j}\) represent a class in \(\rH^{j}(D^{\bullet}/I_{m'}D^{\bullet})\), so that
\(dz\in I_{m'}D^{j+1}\subseteq\frm^{m'}D^{j+1}\) by (\ref{para:group-algebra}). Then
\(dz\in B^{j+1}\cap\frm^{m'}D^{j+1}\subseteq\frm^{m'-c}B^{j+1}\), so \(dz=dw\) for some
\(w\in\frm^{m'-c}D^{j}\). Thus \(z-w\) is a boundary, since \(D^{\bullet}\) is exact in degree \(j\).
Since \(m'-c\geq(2g+c+1)m-c\geq2g(m-1)+1\), (\ref{para:group-algebra}) gives \(\frm^{m'-c}\subseteq I_m\),
so \(w\in I_mD^{j}\). Hence the image of \(z\) in \(D^{j}/I_mD^{j}\) is a boundary, and the map
\(\rH^{j}(K\otimes^{L}_{S}S_{m'})\to\rH^{j}(K\otimes^{L}_{S}S_m)\) is zero for every \(j<0\). By
\ref{item:M1}, this map is the trace map \(\rH^{j}(A,[m']^{*}\bbD F)\to\rH^{j}(A,[m]^{*}\bbD F)\), and by
(\ref{para:transitions}) the pullback map \(\rH^{-j}(A,[m]^{*}F)\to\rH^{-j}(A,[m']^{*}F)\) is its
\(\Fl\)-linear dual, hence also zero.
\end{proof}

\subsection{Characteristic cycles}
\label{sec:ss-cc}

We recall the necessary facts from Saito's construction of characteristic cycles \cite{Sai17}.

\begin{enumerate}[label=\textup{(S\arabic*)}, ref=\textup{(S\arabic*)}]

\item\label{item:S2}\label{item:S4}\label{item:S3}
Let \(X\) be a smooth variety of dimension \(d\) over \(k\) and \(F\in\Dbc(X,\Lambda)\). Saito defined
the characteristic cycle \(\CC(F)\), a \(d\)-cycle on \(T^*X\) with integer coefficients and
conical support \cite[Theorems~5.9 and~5.18]{Sai17}. Moreover, if \(X\) is projective, the index formula
\cite[Theorem~7.13]{Sai17} gives
\[
  \chi(X,F)\ =\ \bigl(\CC(F),T_X^*X\bigr)_{T^*X},
\]
where \(T_X^*X\) is the zero section.

The characteristic cycle is additive in triangles, hence in short exact sequences of perverse
sheaves, and \(\CC(F[1])=-\CC(F)\) \cite[Lemma~5.13.1]{Sai17}.

\item\label{item:S-effectivity} For \(P\in\Perv(X,\Lambda)\),
\cite[Proposition~5.14]{Sai17} shows that the cycle \(\CC(P)\) is effective and its support equals
Beilinson's singular support \(\SSupp(P)\) \cite{B}.

\item\label{item:S5} Let \(C\subseteq T^*X\) be closed and conical, with all irreducible components
of dimension \(d\), and let \(i\colon W\hookrightarrow X\) be a closed immersion with \(W\) smooth.
The immersion is \emph{\(C\)-transversal} if \((W\times_X C)\cap N^*_{W/X}\) is contained in the
zero section, where \(N^*_{W/X}=\ker(T^*X|_W\to T^*W)\) is the conormal bundle.

It is \emph{properly \(C\)-transversal} if, in addition, every irreducible component of
\(W\times_X C\) has dimension \(\dim W\) \cite[Definition~7.1]{Sai17}.

\item\label{item:S7} Let \(P\in\Perv(X,\Lambda)\), and let \(i\colon W\hookrightarrow X\) be a
\(\SSupp(P)\)-transversal closed immersion of codimension \(c\), with \(W\) smooth. By
\cite[Proposition~8.13.1 and Lemma~8.6.5]{Sai17}, we have \(Ri^!P\simeq i^*P[-2c]\), and
\(i^*P[-c]\) and \(Ri^!P[c]\) are perverse.\footnote{The purity isomorphism
\(Ri^!\Lambda\simeq\Lambda[-2c]\), required in Saito's Lemma~8.6.5, follows from
\cite[Expos\'e~XVI, Th\'eor\`eme~3.7 and Remarques~3.10]{ArtinGrothendieckSGA4}.}

\item\label{item:S8}
Let \(p\colon\ol{T^*X}:=\bbP(T^*X\oplus\cO_X)\to X\) be the bundle of lines, and let
\(\xi=c_1(\cO(1))\), where \(\cO(1)\) is the dual of the universal subbundle. Let
\(\ol{\CC(F)}\) denote the closure of the characteristic cycle in \(\ol{T^*X}\). The projective bundle
formula defines unique classes \(\operatorname{cc}_{X,a}(F)\in\CH_a(X)\) by
\[
  [\ol{\CC(F)}]\ =\ \sum_{a=0}^d p^*\operatorname{cc}_{X,a}(F)\,\xi^a
  \quad\text{in }\CH_d(\ol{T^*X}).
\]
Their sum \(\operatorname{cc}_X(F):= \sum_{a=0}^{d}\operatorname{cc}_{X,a}(F) \in \CH_{*}(X)\) is the
\emph{characteristic class} \cite[(6.12) and Definition~6.7]{Sai17}.

The characteristic class is additive in \(\CC(F)\), so
\(\operatorname{cc}_X(F[1])=-\operatorname{cc}_X(F)\) by \ref{item:S3}.
By \cite[Lemma~6.9.1]{Sai17}, its component in \(\CH_0(X)\) is the intersection class of \(\CC(F)\)
with the zero section. Thus, when \(X\) is projective, \ref{item:S4} gives
\[
  \chi(X,F)\ =\ \deg\operatorname{cc}_{X,0}(F).
\]

\item\label{item:S9}
Let \(P\in\Perv(X,\Lambda)\), and let \(i\colon D\hookrightarrow X\) be the inclusion of a smooth
divisor. Assume that \(i\) is properly \(\SSupp(P)\)-transversal.

Let \(i^!\colon\CH_a(X)\to\CH_{a-1}(D)\) be the Gysin homomorphism of the divisor
\cite[\S6.2]{Ful98}, and let \(N_{D/X}\simeq\cO_X(D)|_D\) be its normal bundle. The conormal line bundle
has total Chern class \(c(N_{D/X}^*)=1-c_1(N_{D/X})\). By \cite[Proposition~7.8]{Sai17} and
\ref{item:S3},
\[
  \operatorname{cc}_D(i^*P[-1])
  \ =\ \bigl(1-c_1(N_{D/X})\bigr)^{-1}\cap i^!\operatorname{cc}_X(P).
\]
The first Chern class operator \(c_1(N_{D/X})\cap-\) lowers cycle dimension by one, so the inverse
acts on \(\CH_\bullet(D)\) as the finite sum \(\sum_{j=0}^{\dim D}c_1(N_{D/X})^j\cap-\).
For an \'etale morphism \(f\colon W\to X\), the relative cotangent bundle is zero, so
\cite[Corollary~7.9]{Sai17} gives
\[
  \operatorname{cc}_W(f^*F)\ =\ f^*\operatorname{cc}_X(F),
\]
where \(f^*\) is the flat pullback of cycle classes.

\end{enumerate}

\begin{lem}\label{lem:bertini}
Let \(X\) be a smooth projective variety over \(k\) and \(H\) a very ample divisor on \(X\). Let
\(Y\subseteq X\) be a smooth projective subvariety of dimension \(r\geq1\), and let
\(C\subseteq T^{*}Y\) be a closed conical subset all of whose irreducible components have dimension
\(r\). Then, for every member \(Z\) of a dense open subset of the projective space \(|H|\), the
intersection \(Z\cap Y\) is smooth of dimension \(r-1\) and the inclusion
\(Z\cap Y\hookrightarrow Y\) is properly \(C\)-transversal.
\end{lem}

\begin{proof}
Embed \(X\subseteq\bbP:=\bbP\bigl(\rH^{0}(X,\cO_X(H))^{\vee}\bigr)\) by the very ample \(|H|\), so that
members of \(|H|\) are the hyperplane sections of \(X\). Let \(\pi\colon T^{*}\bbP|_{Y}\to T^{*}Y\) be
the canonical surjection of bundles on \(Y\) and let
\[
  C''\ :=\ \pi^{-1}(C)\ \cup\ T^{*}_{Y}\bbP\ \subseteq\ T^{*}\bbP,
\]
a closed conical subset all of whose irreducible components have dimension \(\dim\bbP\). By
\cite[Lemma~1.3.7.2]{Sai} the hyperplanes \(\mathsf{H}\subset\bbP\) which are properly
\(C''\)-transversal form a dense open subset of \(|H|\). We claim that every such \(\mathsf{H}\)
satisfies the conclusion with \(Z=\mathsf{H}\cap X\). Let \(z\in\mathsf{H}\cap Y\), and choose a local
equation \(f=0\) for \(\mathsf{H}\) near \(z\). Let \(\theta=df_z\in T^{*}_{z}\bbP\), and let
\(\bar\theta\in T^{*}_{z}Y\) be its image under \(\pi\). Then \(\bar\theta=d(f|_Y)_z\).
Transversality of \(\mathsf{H}\) to \(T^{*}_{Y}\bbP\) gives \(\bar\theta\neq0\), so \(\mathsf{H}\cap Y\)
is a smooth divisor in \(Y\) near \(z\)
\cite[Tag~\href{https://stacks.math.columbia.edu/tag/057C}{057C}]{stacks-project}.
Its conormal line at \(z\) is generated by \(\bar\theta\). Transversality of \(\mathsf{H}\) to
\(\pi^{-1}(C)\) gives \(\bar\theta\notin C_{z}\), which is \(C\)-transversality of
\(\mathsf{H}\cap Y\hookrightarrow Y\) at \(z\).

Next we show proper \(C\)-transversality. Note that \(\mathsf{H}\times_{\bbP}\pi^{-1}(C)=\pi^{-1}(C)|_{\mathsf{H}\cap Y}\) is
an affine bundle of relative dimension \(\dim\bbP-r\) over \((\mathsf{H}\cap Y)\times_{Y}C\). Its
irreducible components have dimension at most \(\dim\bbP-1\), being contained in irreducible
components of \(\mathsf{H}\times_{\bbP}C''\), and at least \(\dim\bbP-1\), by Krull's principal ideal
theorem, since \(\mathsf{H}\) is a Cartier divisor in \(\bbP\) and \(\pi^{-1}(C)\) has pure dimension
\(\dim\bbP\). Hence every irreducible component of \((\mathsf{H}\cap Y)\times_{Y}C\) has dimension
\((\dim\bbP-1)-(\dim\bbP-r)=r-1\).
\end{proof}

\section{Euler characteristic bounds}
\label{sec:chi}

This section proves the graded estimate, Proposition~\ref{prop:graded-estimate}, which expresses the
Euler characteristics of the members of a properly transversal complete intersection flag through the
characteristic numbers of a perverse sheaf.

Let \(X\) be a smooth projective variety over \(k\) of dimension \(d\geq1\), and fix a very ample divisor
\(H\) on \(X\).

For a line bundle \(L\), the first Chern class operator is
\(c_1(L)\cap-\colon\CH_a(X)\to\CH_{a-1}(X)\). Let \(c_1(L)\) also denote
its value on \([X]\), regarded as a class in \(\CH^1(X)=\CH_{d-1}(X)\). Let
\(h=c_1(\cO_X(H))\), and let \(D\) also denote the class \(c_1(\cO_X(D))\) of a divisor \(D\).
The degree map \(\deg\colon\CH_0(X)\to\bbZ\) is proper pushforward along \(X\to\Spec k\), and
\(H^d=\deg(h^d\cap[X])\).

Let \(N^1(X)\) denote \(\CH^1(X)\) modulo numerical equivalence
\cite[Definition~19.1.1]{Ful98}, and let \(N^1(X)_{\bbR}=N^1(X)\otimes_{\bbZ}\bbR\). We also denote
the image of \(h\) in \(N^1(X)_{\bbR}\) by \(h\).

For \(1\leq a\leq d\) and \(\alpha\in\CH_a(X)\), the intersection number
\(\deg(D_1\cdots D_a\cap\alpha)\) is symmetric and multilinear in the divisors \(D_1,\dots,D_a\).
If any \(D_i\) is numerically trivial, it has degree zero on the \(1\)-cycle class
\(\bigl(\prod_{j\neq i}D_j\bigr)\cap\alpha\in\CH_1(X)\). Thus the intersection number depends only
on the numerical class of each divisor and extends to a real-valued multilinear form on
\(N^1(X)_{\bbR}\). For \(a=0\), the empty product gives \(\deg(\alpha)\).

Every morphism \(f\colon X\to X\) induces by pullback an endomorphism of \(N^1(X)_{\bbR}\)
\cite[Example~19.1.6]{Ful98}.

For \(F\in\Dbc(X,\Lambda)\) and \(0\leq a\leq d\), let
\(\operatorname{cc}_{X,a}(F)\in\CH_a(X)\) be the component of dimension \(a\) of the characteristic
class of \ref{item:S8}. Define the symmetric multilinear form
\(\gamma_a(F)\colon (N^1(X)_{\bbR})^a\to\bbR\) by
\[
  \gamma_a(F)(D_1,\dots,D_a)\ :=\ \deg\bigl(D_1\cdots D_a\cap\operatorname{cc}_{X,a}(F)\bigr),
\]
and let \(\gamma_a(F)(D):=\gamma_a(F)(D,\dots,D)\) for \(D\in N^1(X)_{\bbR}\). The
characteristic numbers of \(F\) with respect to \(h\) are
\begin{equation}\label{eq:characteristic-numbers}
  c_a(F)\ :=\ \gamma_a(F)(h)\ =\ \deg\Bigl(h^a\cap\operatorname{cc}_{X,a}(F)\Bigr)
  \qquad(0\leq a\leq d).
\end{equation}
By \ref{item:S8} and the index formula \ref{item:S4}, \(c_0(F)=\chi(X,F)\).

\begin{lem}\label{lem:scaling}
Let \(f\colon X\to X\) be \'etale and let \(F\in\Dbc(X,\Lambda)\).
\begin{enumerate}
\item \(\operatorname{cc}_{X}(f^{*}F)=f^{*}\operatorname{cc}_{X}(F)\) in \(\CH_{\bullet}(X)\), where
\(f^{*}\) is the flat pullback of cycle classes along \(f\).
\item For \(0\leq a\leq d\) and \(D_{1},\dots,D_{a}\in N^{1}(X)_{\bbR}\),
\[
  \gamma_{a}(f^{*}F)\bigl(f^{*}D_{1},\dots,f^{*}D_{a}\bigr)\ =\ \deg(f)\,\gamma_{a}(F)(D_{1},\dots,D_{a}).
\]
In particular, if \(h'\in N^{1}(X)_{\bbR}\) satisfies \(f^{*}h'=h\), then
\(c_a(f^{*}F)=\deg(f)\,\gamma_a(F)(h')\) for \(0\leq a\leq d\), where the characteristic numbers
on the left are computed with respect to \(h\).
\end{enumerate}
\end{lem}

\begin{proof}
(1) is the \'etale pullback formula of \ref{item:S9}.

(2) Both sides are multilinear in the \(D_{i}\), so we may take the \(D_{i}\) to be classes of
divisors. Chern classes commute with flat pullback \cite[Proposition~2.5(d)]{Ful98}, so (1) gives
\[
  f^{*}D_{1}\cdots f^{*}D_{a}\cap\operatorname{cc}_{X,a}(f^{*}F)
  \ =\ f^{*}\Bigl(D_{1}\cdots D_{a}\cap\operatorname{cc}_{X,a}(F)\Bigr),
\]
and taking degrees gives the identity, since flat pullback along \(f\) multiplies the degree of a
zero-cycle by \(\deg(f)\) \cite[Example~1.7.4]{Ful98}. The last statement follows by taking
\(D_{1}=\dots=D_{a}=h'\).
\end{proof}

\begin{lem}\label{lem:flags}
Let \(P\in\Perv(X,\Lambda)\). There exists a flag of smooth closed subschemes
\[
  Y_{0}\subset Y_{1}\subset\dots\subset Y_{d}=X,
\]
where \(Y_r\) is a complete intersection of \(d-r\) members of \(|H|\).
Moreover, if we define \(Q_r:=P|^*_{Y_r}[r-d]\), then each inclusion
\(\iota_r\colon Y_{r-1}\hookrightarrow Y_r\) is properly \(\SSupp(Q_r)\)-transversal and
\[
  Q_{r-1}=\iota_r^*Q_r[-1]\simeq\iota_r^!Q_r[1]
  \qquad(1\leq r\leq d).
\]
\end{lem}

\begin{proof}
We proceed by descending induction on \(r\). Having constructed \(Y_r\) and the perverse sheaf \(Q_r\), let
\(C_r:=\SSupp(Q_r)\subseteq T^*Y_r\). By \ref{item:S2} and \ref{item:S-effectivity}, \(C_r\) is closed
and conical with all components of dimension \(r\).
Lemma~\ref{lem:bertini} provides a member \(Z\in|H|\) with \(Y_{r-1}:=Z\cap Y_r\) smooth and
\(\iota_r\) properly \(C_r\)-transversal. Then \(Q_{r-1}=\iota_r^*Q_r[-1]\) is perverse by
\ref{item:S7}, and the induction continues.
By construction for all \(r\), \(Q_r=P|^*_{Y_r}[r-d]\).
\end{proof}

\begin{prop}\label{prop:graded-estimate}
Let \(P\in\Perv(X,\Lambda)\) and let the flag be as in Lemma~\ref{lem:flags}. Then for
\(0\leq r\leq d\)
\[
  \chi\bigl(Y_{r},\,Q_{r}\bigr)
  \ =\
  \sum_{a=d-r}^{d}\binom{a-1}{a-d+r}\,c_a(P),
\]
where the characteristic numbers \(c_a(P)\) are computed with respect to \(h\), and
\(\binom{-1}{0}=1\).

Consequently, if
\(f\colon X\to X\) is \'etale and \(h'\in N^{1}(X)_{\bbR}\) satisfies \(f^{*}h'=h\), then
for any flag \((Y_r)\) and sheaves \(Q_r\) constructed as in Lemma~\ref{lem:flags} for \(f^{*}P\),
\[
  \bigl|\chi\bigl(Y_{r},\,Q_{r}\bigr)\bigr|
  \ \leq\ 2^{d}\deg(f)\sum_{a=d-r}^{d}\bigl|\gamma_{a}(P)(h')\bigr|.
\]
\end{prop}

\begin{proof}
Let \(j_r\colon Y_r\hookrightarrow X\) be the inclusion, and let \(h_r=j_r^*h\) and
\(\alpha_r=(j_r)_*\operatorname{cc}_{Y_r}(Q_r)\in\CH_\bullet(X)\). For \(1\leq s\leq d\), the
normal bundle of \(\iota_s\colon Y_{s-1}\hookrightarrow Y_s\) is \(\cO_X(H)|_{Y_{s-1}}\).
The divisor pullback formula \ref{item:S9} and \ref{item:S3} give
\[
  \operatorname{cc}_{Y_{s-1}}(Q_{s-1})
  \ =\ (1-h_{s-1})^{-1}\cap\iota_s^!\operatorname{cc}_{Y_s}(Q_s).
\]
For \(\beta\in\CH_\bullet(Y_s)\), the identity
\((\iota_s)_*\iota_s^!\beta=h_s\cap\beta\) and the projection formula
\cite[Propositions~2.6(b) and~2.5(c)]{Ful98} therefore give
\[
  \alpha_{s-1}\ =\ \frac{h}{1-h}\cap\alpha_s.
\]
Since \(\alpha_d=\operatorname{cc}_X(P)\), iteration yields
\[
  (j_r)_*\operatorname{cc}_{Y_r}(Q_r)
  \ =\ \frac{h^{d-r}}{(1-h)^{d-r}}\cap\operatorname{cc}_X(P).
\]
Let \(c=d-r\). For \(c>0\), we expand \((1-h)^{-c}\) and take the degree of the zero-dimensional
component. By \ref{item:S8} and \ref{item:S4}, this gives
\begin{equation}\label{eq:chi-refined}
  \chi(Y_r,Q_r)\ =\ \sum_{a=c}^d\binom{a-1}{a-c}\,
  \deg\Bigl(h^a\cap\operatorname{cc}_{X,a}(P)\Bigr).
\end{equation}
This holds also for \(c=0\), when only \(a=0\) contributes, with \(\binom{-1}{0}=1\).
By \eqref{eq:characteristic-numbers}, \(\deg(h^a\cap\operatorname{cc}_{X,a}(P))=c_a(P)\), so this
is the required formula.

For the bound, we apply the formula to \(f^{*}P\), which is perverse
since \(f\) is \'etale, use Lemma~\ref{lem:scaling}(2), and bound the binomial coefficients by \(2^{d}\).
\end{proof}

\section{Proofs of Theorems \ref{thm:A}, \ref{thm:B} and \ref{thm:C}}
\label{sec:proofs}

This section combines the estimate of \S\ref{sec:chi} with Artin vanishing and proves
Theorem~\ref{thm:A}. Proposition~\ref{prop:HS-machine} then converts these Betti bounds into codimension
estimates for the completed Fourier--Mellin transform \(K\) of \S\ref{sec:mellin} and its dual
\(\RHom_S(K,S)\). These give Theorem~\ref{thm:C} and, via Lemma~\ref{lem:tower-vanishing}, the vanishing
of Theorem~\ref{thm:B}.

\subsection{Proof of Theorem \ref{thm:A}}
\label{sec:recursion}

We continue using the notation \(X\), \(d\), \(H\), \(h\), \(H^d\) and \(N^1(X)_{\bbR}\) of
\S\ref{sec:chi}.

\begin{lem}\label{lem:one-cut}
Let \(Y\) be a smooth projective variety, \(Q\in\Perv(Y,\Lambda)\), and let \(Z\subset Y\) be an ample
divisor, smooth and properly \(\SSupp(Q)\)-transversal, with perverse restriction
\(Q_Z:=Q|^{*}_{Z}[-1]\in\Perv(Z,\Lambda)\). Then
\begin{enumerate}
\item \(\rH^{a}(Y,Q)\simeq\rH^{a-1}(Z,\,Q_Z)\) for \(a\geq2\), and \(\rH^{1}(Y,Q)\) is a
quotient of \(\rH^{0}(Z,\,Q_Z)\).
\item \(\rH^{a}(Y,Q)\simeq\rH^{a+1}(Z,Q_Z)\) for \(a\leq-2\), and \(\rH^{-1}(Y,Q)\) is a
subobject of \(\rH^{0}(Z,Q_Z)\).
\end{enumerate}
In particular
\[
  \sum_{a\neq0}\dim\rH^{a}(Y,Q)
  \leq 2\dim\rH^{0}(Z,Q_Z)+\sum_{a\neq0}\dim\rH^{a}(Z,Q_Z).
\]
\end{lem}

\begin{proof}
Let \(j\colon U=Y\setminus Z\hookrightarrow Y\) and \(i\colon Z\hookrightarrow Y\). Since \(Z\) is ample,
\(U\) is affine, so Artin vanishing \cite[Th\'eor\`eme~4.1.1]{BBDG18} gives
\(\rH^{a}(U,j^{*}Q)=0\) for \(a>0\), and its Verdier dual gives
\(\rH^{a}_{c}(U,j^{*}Q)=0\) for \(a<0\). The triangle \(i_{*}i^{!}Q\to Q\to Rj_{*}j^{*}Q\) and the
first vanishing give isomorphisms \(\rH^{a}(Z,i^{!}Q)\xrightarrow{\sim}\rH^{a}(Y,Q)\) for \(a\geq2\) together
with a surjection in degree \(1\). The triangle \(j_{!}j^{*}Q\to Q\to i_{*}i^{*}Q\) and the second
vanishing give isomorphisms \(\rH^{a}(Y,Q)\xrightarrow{\sim}\rH^{a}(Z,i^{*}Q)\) for \(a\leq-2\) together
with an injection in degree \(-1\).

Proper transversality gives \(i^{!}Q\simeq i^{*}Q[-2]\) and makes \(Q_Z=i^{*}Q[-1]\) perverse by
\ref{item:S7}, thus implying (1) and (2). The inequality follows by summation.
\end{proof}

For \(P\in\Perv(X,\Lambda)\), let
\begin{equation}\label{eq:sigma-definition}
  \sigma(P):=\max_{\bar x\to X}\sum_a\dim_{\Lambda}\cH^a(P)_{\bar x},
\end{equation}
where the maximum ranges over geometric points and is finite by constructibility. For an \'etale
self-map \(f\colon X\to X\), surjectivity and invariance of geometric stalks under pullback give
\begin{equation}\label{eq:sigma-pullback}
  \sigma((f^n)^*P)=\sigma(P)\quad(n\geq0).
\end{equation}

\begin{lem}\label{prop:recursion}
Fix \(P\in\Perv(X,\Lambda)\) and a flag as in Lemma~\ref{lem:flags}. Let
\(B_{r}:=\sum_{a}\dim\rH^{a}(Y_{r},Q_{r})\). Then
\[
  B_{r}\ \leq\ \bigl|\chi(Y_{r},Q_{r})\bigr|\ +\ 4\,B_{r-1}
  \qquad(1\leq r\leq d),
  \qquad
  B_{0}\ \leq\ H^{d}\,\sigma(P),
\]
and consequently
\[
  B_{r}\ \leq\ \sum_{s=1}^{r}4^{\,r-s}\bigl|\chi(Y_{s},Q_{s})\bigr|\ +\ 4^{r}H^{d}\,\sigma(P)
  \qquad(0\leq r\leq d).
\]
\end{lem}

\begin{proof}
Fix \(1\leq r\leq d\) and let \(Y=Y_{r}\), \(Q=Q_{r}\) and \(Z=Y_{r-1}\). By Lemma~\ref{lem:flags},
\(Z\) is a smooth ample divisor in the smooth projective variety \(Y\), it is properly
\(\SSupp(Q)\)-transversal, and \(Q|^{*}_{Z}[-1]=Q_{r-1}\) is perverse, so Lemma~\ref{lem:one-cut}
gives \(\sum_{a\neq0}\dim\rH^{a}(Y,Q)\leq2B_{r-1}\).
Thus
\[
  B_{r}=\dim\rH^{0}(Y,Q)+\sum_{a\neq0}\dim\rH^{a}(Y,Q)
  \leq\bigl|\chi(Y,Q)\bigr|+2\sum_{a\neq0}\dim\rH^{a}(Y,Q)
  \leq\bigl|\chi(Y,Q)\bigr|+4B_{r-1}.
\]
For the base case, \(Y_0\) is a zero-dimensional complete intersection of \(d\) members of \(|H|\),
so \(\deg[Y_0]=H^d\). Since \(Y_0\) is smooth and \(k\) is algebraically closed, it consists of
\(H^d\) reduced points. Moreover, \(Q_0=P|^*_{Y_0}[-d]\), so each point contributes at most
\(\sigma(P)\) to \(B_0\) by \eqref{eq:sigma-definition}. Hence \(B_0\leq H^d\sigma(P)\).
Iterating the recursion gives the last inequality.
\end{proof}

\begin{prop}\label{prop:finite-level}
Let \(f\colon X\to X\) be \'etale and let \(h'\in N^{1}(X)_{\bbR}\) satisfy \(f^{*}h'=h\).
Then for every \(P\in\Perv(X,\Lambda)\) and every \(i\geq0\)
\[
  \dim_{\Lambda}\rH^{i}\bigl(X,f^{*}P\bigr)
  \ \leq\ 4^{d}\Bigl(H^{d}\,\sigma(P)\ +\ 2^{d}\deg(f)\sum_{a=i}^{d}\bigl|\gamma_{a}(P)(h')\bigr|\Bigr),
\]
and \(\dim_{\Lambda}\rH^{-i}(X,f^{*}P)\) is bounded by the same expression with \(\bbD P\) in place
of \(P\).
\end{prop}

\begin{proof}
Since \(f\) is \'etale, \(f^*P\) is perverse, so \(\rH^i(X,f^*P)=0\) for \(i>d\) \cite[4.2.4]{BBDG18}, and we
may assume \(i\leq d\). Choose a flag \(Y_0\subset\cdots\subset Y_d=X\), cut out by members of \(|H|\), as in
Lemma~\ref{lem:flags} for \(f^*P\), and let
\[
  Q_r:=(f^*P)|^*_{Y_r}[r-d],\qquad
  B_r:=\sum_a\dim\rH^a(Y_r,Q_r).
\]
By \eqref{eq:sigma-pullback}, the base case of Lemma~\ref{prop:recursion} gives
\[
  B_0\leq H^d\sigma(f^*P)=H^d\sigma(P),
\]
independently of \(f\) and of the chosen flag. For \(0\leq r\leq d\), 
Lemma~\ref{prop:recursion} and the Euler characteristic bound of
Proposition~\ref{prop:graded-estimate} give
\begin{align*}
  B_r&\leq 4^rH^d\sigma(P)
  +2^d\deg(f)\sum_{s=1}^r4^{r-s}\sum_{a=d-s}^d|\gamma_a(P)(h')|\\
  &\leq 4^r\Bigl(H^d\sigma(P)+2^d\deg(f)\sum_{a=d-r}^d|\gamma_a(P)(h')|\Bigr).
\end{align*}
For the second inequality, each inner sum is at most
\(\sum_{a=d-r}^d|\gamma_a(P)(h')|\), since \(s\leq r\), and
\(\sum_{s=1}^r4^{r-s}=(4^r-1)/3\leq4^r\).

For \(1\leq i\leq d\), Lemma~\ref{lem:one-cut}(1) gives
\[
  \rH^0(Y_{d-i},Q_{d-i})\twoheadrightarrow
  \rH^1(Y_{d-i+1},Q_{d-i+1})\xrightarrow{\sim}\rH^i(X,f^*P),
\]
where the last arrow is the composite of \(i-1\) isomorphisms. Thus
\(\dim\rH^i(X,f^*P)\leq B_{d-i}\). For \(i=0\), the same inequality is
\(\dim\rH^0(X,f^*P)\leq B_d\). Taking \(r=d-i\) in the bound for \(B_r\) and using
\(4^{d-i}\leq4^d\) proves the estimate for \(0\leq i\leq d\). Finally the bound for \(\rH^{-i}(X,f^{*}P)\)
follows from Verdier duality.

\end{proof}

Proposition~\ref{prop:finite-level} gives the following theorem.

\begin{thm}\label{thm:polarized}
Let \(X\) be a smooth projective variety of dimension \(d\) over \(k\) with a very ample divisor
\(H\), and let \(f\colon X\to X\) be an \'etale morphism with \(f^{*}h=qh\) in
\(N^1(X)_{\bbR}\) for some \(q\geq1\). Then for every \(P\in\Perv(X,\Lambda)\) there is a constant
\(C_{P}\), depending only on \(P\) and \(H\), such that for all \(i\in\bbZ\)
\[
  \dim_{\Lambda}\rH^{i}\bigl(X,f^{*}P\bigr)\ \leq\ C_{P}\,q^{d-|i|}.
\]
\end{thm}

\begin{proof}
By Verdier duality and the vanishing of cohomology outside \([-d,d]\), it suffices to prove the
estimate for \(0\leq i\leq d\). Taking top intersection numbers gives
\((\deg f)H^d=(f^{*}H)^d=q^dH^d\). Since \(H^d>0\), we have \(\deg f=q^d\).
Proposition~\ref{prop:finite-level} applies with \(h'=q^{-1}h\), since \(f^*h'=h\). By multilinearity,
\(\gamma_a(P)(h')=q^{-a}c_a(P)\), so
\begin{align*}
  \dim_{\Lambda}\rH^{i}\bigl(X,f^{*}P\bigr)
  \ &\leq\ 4^{d}\Bigl(H^{d}\sigma(P)+2^{d}\sum_{a=i}^{d}q^{d-a}\bigl|c_a(P)\bigr|\Bigr)\\
  &\leq\ 4^{d}\Bigl(H^{d}\sigma(P)+2^{d}\sum_{j=0}^{d}\bigl|c_{j}(P)\bigr|\Bigr)\,q^{d-i}.
\end{align*}
Taking \(C_P\) to be the larger of the constants for \(P\) and \(\bbD P\) proves the theorem.
\end{proof}

\begin{proof}[Proof of Theorem~\ref{thm:A}]
Choose a very ample divisor \(H\) on \(A\), and let \(X=A\) and \(d=g\).
Since \([-1]^{*}h=h\) in \(N^1(A)_{\bbR}\), the formula for pullback by multiplication
\cite[Chapter~II, \S6, Corollary~3, and \S8(iv)]{Mum08} gives \([m]^{*}h=m^{2}h\).
For \(m\) invertible in \(k\), the morphism \([m]\) is \'etale, so Theorem~\ref{thm:polarized}
applies with \(f=[m]\) and \(q=m^2\).
The resulting constant depends only on \(P\) and \(H\). One may take \(C_{P}\) to be the larger of
\(4^{g}(H^{g}\sigma(P)+2^{g}\sum_{j=0}^{g}|c_j(P)|)\) and the same expression for \(\bbD P\).
\end{proof}

\begin{rem}\label{rem:polarized-endomorphisms}
The hypothesis of Theorem~\ref{thm:polarized} for \(q>1\) is restrictive. When \(k=\bbC\), it implies that
\(X\) admits a finite \'etale cover by an abelian variety \cite[Lemma~2.3 and
Theorem~3.3]{NakayamaZhang10}. In positive characteristic, Rai \cite[Theorem~1.1]{Rai26} proves the same
conclusion for a smooth projective variety of non-negative Kodaira dimension that admits a separable
polarized endomorphism and has virtually abelian \'etale fundamental group.
\end{rem}

\begin{rem}\label{rem:BBDG-comparison}
For a product \(A=\prod_{a=1}^{g}E_a\) of elliptic curves, the bound of Theorem~\ref{thm:A} follows from
the proof of \cite[Proposition~4.5.1]{BBDG18}. 
\end{rem}

\begin{rem}\label{rem:katz-cc}
See also \cite{PanZhangZhang26} and \cite{HuTeyssier25}, where Katz's argument is used in combination with
characteristic cycles to bound Betti numbers.
\end{rem}

\subsection{Proofs of Theorems \ref{thm:B} and \ref{thm:C}}
\label{sec:proofs-bc}

We use \(S=\Fl[[x_1,\dots,x_{2g}]]\) and \(I_m=(x_1^m,\dots,x_{2g}^m)\) for \(m=\ell^n\), as in
(\ref{para:group-algebra}). The following proposition converts length bounds at finite level into
codimension bounds.

\begin{prop}\label{prop:HS-machine}
Let \(\cK\) be a perfect complex over \(S\), and suppose \(\cK\in D^{\leq g}(S)\) and that there is a constant
\(C\) with
\[
  \lgth_{S}\ \rH^{i}\bigl(\cK\otimes^{L}_{S}S/I_m\bigr)\ \leq\ C\,m^{2(g-i)}
  \qquad\text{for all } i\geq 0 \text{ and } m=\ell^n.
\]
Then \(\dim\Supp\rH^{i}(\cK)\leq 2(g-i)\), that is \(\codim\Supp\rH^{i}(\cK)\geq2i\), for all
\(i\geq0\).
\end{prop}

\begin{proof}
Let \(M^{b}=\rH^{b}(\cK)\), finite \(S\)-modules vanishing for \(b>g\). We argue by descending
induction on \(i\), the statement being vacuous for \(i>g\). The hyper-Tor spectral sequence
\[
  E_2^{-a,b}\ =\ \Tor^{S}_{a}\bigl(M^{b},\,S/I_m\bigr)\ \Longrightarrow\
  \rH^{b-a}\bigl(\cK\otimes^{L}_{S}S/I_m\bigr)
\]
has \(E_2^{0,i}=M^{i}/I_mM^{i}\) on its edge. The differentials into this spot originate from
\(E_r^{-r,\,i+r-1}\) with \(r\geq2\). These terms vanish for \(r>g-i+1\).

For the remaining terms, the inductive hypothesis gives
\(\dim\Supp M^{i+r-1}\leq 2(g-i-r+1)\leq 2(g-i)-2\).
Since \(x_1^m,\dots,x_{2g}^m\) is a regular sequence of \(S\), the Tor groups on the \(E_2\)-page are
its Koszul homology on the modules \(M^b\). For every nonzero \(M^b\), \cite[Corollary~7]{GR86} bounds these Koszul homology lengths by a polynomial
in \(m\) of degree at most \(\dim\Supp M^b\) (see also \cite[Proposition~6.14]{Mou24} for a multigraded
refinement). The inductive hypothesis therefore gives a constant \(D_i\geq0\), independent of \(m\),
such that the sum of the lengths of these source terms is at most \(D_i m^{2(g-i)-2}\).

The abutment in total degree \(i\) has length at most \(C\,m^{2(g-i)}\) by hypothesis. There are no
outgoing differentials from \(E_r^{0,i}\), so
\[
  \lgth_S\ \bigl(M^{i}/I_mM^{i}\bigr)
  \ \leq\ C\,m^{2(g-i)}+D_i m^{2(g-i)-2}\ \leq\ (C+D_i)m^{2(g-i)}.
\]
If \(M^i=0\), the desired bound is immediate. Otherwise, since \(I_m\subseteq\frm^m\), the same
upper bound holds for \(\lgth_S(M^i/\frm^mM^i)\). This function agrees for large \(m\) with its
Hilbert--Samuel polynomial, whose degree is \(\dim\Supp M^i\) and whose leading coefficient is
positive. Comparing growth along \(m=\ell^n\) gives \(\dim\Supp M^i\leq2(g-i)\).
\end{proof}

\begin{proof}[Proof of Theorem~\ref{thm:B}]

Let \(F\) be a constructible sheaf and let \(s=\dim\Supp F\). Then \(F[s]\in{}^{p}D^{\leq0}(A,\Lambda)\), so
\(G=\bbD(F[s])\) lies in \({}^{p}D^{\geq0}(A,\Lambda)\), and its finite filtration by perverse truncations has
graded pieces \(\bbD P_j[-j]\) with \(P_j={}^{p}\rH^{-j}(F[s])\in\Perv(A,\Lambda)\) and \(j\geq0\). The functor
\(\widehat{\FM}_A\) is exact, being a composition of derived functors, and \(D^{\geq0}(S)\) is stable under
extensions. Since \(\rH^{i}(A,[\ell^{n}]^{*}F[s])=\rH^{i+s}(A,[\ell^{n}]^{*}F)\),
Lemma~\ref{lem:tower-vanishing} applied to \(F[s]\) shows that both statements follow once
\(\widehat{\FM}_A(\bbD P)\in D^{\geq0}(S)\) for every \(P\in\Perv(A,\Lambda)\).

Let  \(K=\widehat{\FM}_A(\bbD P)\) and \(\cK'=\RHom_S(K,S)\). By
\ref{item:M2} for the perverse sheaf \(\bbD P\), the complexes \(K\) and \(\cK'\) are perfect and belong to
\(D^{[-g,g]}(S)\). Thus
\[
  \cK'\otimes^{L}_{S}S_m\ \simeq\ \RHom_{S_m}\bigl(K\otimes^{L}_{S}S_m,\,S_m\bigr)
  \ \simeq\ \RHom_{S_m}\bigl(R\Gamma(A,[m]^{*}\bbD P),\,S_m\bigr)
\]
by \ref{item:M1}.

The ring \(S_m=S/I_m\simeq\Fl[x_1,\dots,x_{2g}]/(x_1^m,\dots,x_{2g}^m)\) is a local complete
intersection of dimension \(0\), hence Gorenstein. Its dualizing module
\(\omega_{S_m/\Fl}=\Hom_{\Fl}(S_m,\Fl)\) is free of rank one
\cite[Tag~\href{https://stacks.math.columbia.edu/tag/0DWL}{0DWL}]{stacks-project}.
Choosing an isomorphism \(\omega_{S_m/\Fl}\simeq S_m\), adjunction gives the Artinian Gorenstein
duality isomorphism \(\Hom_{S_m}(N,S_m)\simeq\Hom_{\Fl}(N,\Fl)\), functorially in the \(S_m\)-module
\(N\). In particular, \(\Hom_{S_m}(-,S_m)\) is exact.

Since the residue field of \(S\) is \(\Fl\), a finite \(S\)-module annihilated by a power of \(\frm\)
has a composition series with all factors \(\Fl\), so its length is its \(\Fl\)-dimension. Thus
\(\Hom_{S_m}(-,S_m)\) also preserves lengths of finite \(S_m\)-modules. This converts the bounds of
Theorem~\ref{thm:A} into the length bounds required by Proposition~\ref{prop:HS-machine}.
Hence
\[
  \rH^{j}\bigl(\cK'\otimes^{L}_{S}S_m\bigr)\ \simeq\ \Hom_{S_m}\bigl(\rH^{-j}(A,[m]^{*}\bbD P),S_m\bigr)
\]
has length \(\dim_{\Fl}\rH^{-j}(A,[m]^{*}\bbD P)\leq C_{\bbD P}\,m^{2(g-j)}\) for \(j\geq0\) by
Theorem~\ref{thm:A}.

Proposition~\ref{prop:HS-machine} applied to \(\cK'\) yields
\(\codim\Supp\rH^{j}(\cK')\geq2j\geq j\) for \(j\geq0\), so \(K=\widehat{\FM}_A(\bbD P)\in D^{\geq0}(S)\) by
\ref{item:M3}, for every perverse \(P\), as required.
\end{proof}

\begin{rem}\label{rem:dual-codimension}
For \(P\in\Perv(A,\Lambda)\), let
\(\cK'=\RHom_S(\widehat{\FM}_A(\bbD P),S)\). Then the above proof implies that 
\[
  \codim\Supp\rH^{j}(\cK')\geq 2j \qquad\text{for all }j\geq0.
\]
\end{rem}

\begin{proof}[Proof of Theorem~\ref{thm:C}]
Let \(P\) be perverse and \(\cK=\widehat{\FM}_A(P)\). By \ref{item:M1} and \ref{item:M2}, the complex \(\cK\) is
perfect and belongs to \(D^{[-g,g]}(S)\). Its derived restriction to \(S_m\) is
\(R\Gamma(A,[m]^{*}P)\) by \ref{item:M1}. Theorem~\ref{thm:A} combined with
Proposition~\ref{prop:HS-machine} gives
\(\codim\Supp\rH^{i}(\widehat{\FM}_A(P))\geq2i\) for \(i\geq0\).\footnote{
For elliptic curves \(E_1,\dots,E_g\), take \(A=E_1\times\dots\times E_g\) and
\(i_B:B=E_1\times\dots\times E_b\hookrightarrow A\), with
\(0\leq b\leq g\). Then
\[
  \widehat{\FM}_A(i_{B*}\Fl[b])\simeq(S/\frm_BS)[-b],
\]
where \(\frm_BS\) is generated by \([\gamma]-1\) for \(\gamma\in\Tl B\). Its cohomology in
degree \(b\) has support of codimension \(2b\), so the bound is sharp.}
\end{proof}

\begin{cor}\label{cor:connectivity}
Let \(P\in\Perv(A,\Lambda)\) and \(\cK=\widehat{\FM}_A(P)\).
\begin{enumerate}
\item The complexes \(\cK\) and \(\RHom_S(\cK,S)\) lie in \(D^{\geq0}(S)\).
\item There is a closed subset \(Z\subseteq\Spec S\) of codimension at least \(2\) such that on
\(\Spec S\setminus Z\) the complex \(\cK\) is a locally free module of rank \(\chi(A,P)\) concentrated
in degree \(0\). In particular \(\chi(A,P)\geq0\).
\end{enumerate}
\end{cor}

\begin{proof}
(1) Theorem~\ref{thm:C} and the criterion \ref{item:M3} give
\(\RHom_S(\cK,S)\in D^{\geq0}(S)\). Applying Remark~\ref{rem:dual-codimension} to
\(\bbD P\) and using the same criterion gives \(\cK\in D^{\geq0}(S)\).

(2) Let \(Z\) be the union of the supports of \(\rH^{i}(\cK)\) and of \(\rH^{i}(\RHom_S(\cK,S))\)
for \(i\geq1\), of codimension at least \(2\) by Theorem~\ref{thm:C} and
Remark~\ref{rem:dual-codimension} as applied in (1). Since \(\cK\) is perfect by
Lemma~\ref{lem:FM-finiteness} and both \(\cK\) and its dual are concentrated in degree \(0\) on
\(U=\Spec S\setminus Z\), the restriction \(\cK|_U\) is a locally free module. Its rank is the Euler
characteristic of \(\cK\), which equals \(\chi(A,P)\) by \ref{item:M1} at \(m=1\).
\end{proof}

Let \(\kappa\) be a finite extension of \(\Fl\) and let \(S_{\kappa}=\kappa[[\Tl A]]\). The constructions and
isomorphisms of \S\ref{sec:mellin}, including Lemmas~\ref{lem:finite-level}
and~\ref{lem:FM-finiteness} and the cofinality of the ideals \(I_m\) with the powers of the augmentation
ideal, apply verbatim over the finite field \(\kappa\) and define \(\widehat{\FM}_A(Q)\in D(S_{\kappa})\),
the completed stalk at the trivial character of the \(\kappa\)-linear Fourier--Mellin transform of
\(Q\), for \(Q\in\Dbc(A,\kappa)\). The following corollary of Theorem~\ref{thm:C} is used in
\S\ref{sec:ell-adic} to obtain the \(\ell\)-adic estimates.

\begin{cor}\label{cor:kappa-codim}
Let \(\kappa\) be a finite extension of \(\Fl\) and let \(Q\in\Perv(A,\kappa)\). Then
\[
  \codim_{\Spec S_{\kappa}}\ \Supp\ \rH^{i}\bigl(\widehat{\FM}_A(Q)\bigr)\ \geq\ 2i
  \qquad\text{for all } i\geq 0.
\]
\end{cor}

\begin{proof}
Let \(S_{0}=\Fl[[\Tl A]]\) and regard \(Q\) also as a perverse sheaf with \(\Fl\)-coefficients.
The extension \(S_{0}\to S_{\kappa}=S_{0}\otimes_{\Fl}\kappa\) is finite free, and
\(\Ldot_{S_{\kappa}}=\Ldot_{S_{0}}\otimes_{\Fl}\kappa\) level by level. Hence
\[
  Q\otimes_{\kappa}\Ldot_{S_{\kappa}}
  \ \simeq\
  Q\otimes_{\Fl}\Ldot_{S_{0}}
\]
as sheaves of \(S_{\kappa}\)-modules. Taking \(R\Gamma\) shows that \(\widehat{\FM}_A(Q)\), viewed as an
\(S_{0}\)-complex, is the completed Fourier--Mellin transform of the perverse \(\Fl\)-sheaf \(Q\).
Theorem~\ref{thm:C} therefore gives
\(\codim_{\Spec S_{0}}\Supp\rH^{i}(\widehat{\FM}_A(Q))\geq 2i\) for \(i\geq0\). Every prime of
\(S_{\kappa}\) contracts to a prime of \(S_{0}\) of the same height, since \(S_{0}\to
S_{\kappa}\) is finite free, and the support over \(S_{0}\) is the image of the support
over \(S_{\kappa}\). The estimate over \(S_{\kappa}\) follows.
\end{proof}

\begin{rem}\label{rem:characteristic}
The proofs of Theorems~\ref{thm:A}, \ref{thm:B} and~\ref{thm:C} apply in every characteristic
different from \(\ell\). Over \(\bbC\), this gives a proof of Theorem~\ref{thm:B} independent of
Artin vanishing on the Stein universal cover used in \cite[Proposition~2.7]{BhattSchnellScholze2018}.
\end{rem}

\section{The Fourier--Mellin transform on character varieties}
\label{sec:ell-adic}

In \S\S\ref{sec:integral-lattices} and~\ref{sec:algebraic-model}, let \(E\) be a finite extension of
\(\Ql\) inside an algebraic closure \(\overline{\bbQ}_{\ell}\), with ring of integers \(\cO\),
uniformizer \(\varpi\) and residue field \(\kappa\), and let
\[
  R=\cO[[\Tl A]],\qquad
  \overline R=R/\varpi R=\kappa[[\Tl A]]=S_{\kappa},\qquad
  \mathfrak R=R\otimes_{\cO}\overline{\bbQ}_{\ell}.
\]
A \(\Zl\)-basis \(\gamma_1,\ldots,\gamma_{2g}\) of \(\Tl A\) identifies \(R\) with
\(\cO[[x_{1},\ldots,x_{2g}]]\) via \([\gamma_{j}]=1+x_{j}\). We use the middle perverse \(t\)-structure
for \(\cO\)-coefficients.

In this section we prove Corollary~\ref{cor:D1} by applying Theorem~\ref{thm:C} to the modular
reduction of a torsion-free integral form.

\subsection{Integral forms and specialization}
\label{sec:integral-lattices}

We include the proof of the following standard lemma for ease of reference.

\begin{lem}\label{lem:special-generic-codim}
Let \(\cO\) be a complete discrete valuation ring with uniformizer \(\varpi\), and let
\[
  R=\cO[[x_1,\ldots,x_d]],\qquad \overline R=R/\varpi R.
\]
Let \(K\in D^b_{\mathrm{coh}}(R)\), fix \(i\in\bbZ\), and suppose that
\[
  \codim_{\Spec\overline R}\Supp
  \rH^i\bigl(K\otimes_R^L\overline R\bigr)\ \geq\ c.
\]
Then
\[
  \codim_{\Spec R}\Supp\rH^i(K)\ \geq\ c
  \quad\text{and}\quad
  \codim_{\Spec R[1/\varpi]}\Supp\rH^i\bigl(K[1/\varpi]\bigr)\ \geq\ c.
\]
Moreover, every component of \(\Supp\rH^i(K)\) contained in the special fibre
\(\Spec\overline R\) has codimension at least \(c+1\) in \(\Spec R\).
\end{lem}

\begin{proof}
Let \(M=\rH^i(K)\). The reduction triangle gives an injection
\[
  M/\varpi M\ \lhook\joinrel\longrightarrow\
  \rH^i\bigl(K\otimes_R^L\overline R\bigr).
\]
Let \(\mathfrak p\in\Supp M\). If \(\varpi\in\mathfrak p\), then Nakayama's lemma gives
\((M/\varpi M)_{\mathfrak p}\neq0\), and hence
\[
  \operatorname{ht}_{\overline R}(\mathfrak p/\varpi)
  =\operatorname{ht}_R(\mathfrak p)-1\ \geq\ c.
\]
Thus \(\operatorname{ht}_R(\mathfrak p)\geq c+1\).

Suppose that \(\varpi\notin\mathfrak p\), and let \(\mathfrak q\) be minimal over
\(\mathfrak p+(\varpi)\). Then \(M_{\mathfrak q}\neq0\), so Nakayama's lemma and the
injection show that \(\mathfrak q/\varpi\) belongs to the support on the special fibre. Since
\(R/\mathfrak p\) is a domain and the image of \(\varpi\) is nonzero, the principal ideal theorem and
catenarity of \(R\) give
\[
  \operatorname{ht}_R(\mathfrak q)=\operatorname{ht}_R(\mathfrak p)+1,
  \qquad
  \operatorname{ht}_{\overline R}(\mathfrak q/\varpi)
  =\operatorname{ht}_R(\mathfrak p).
\]
It follows that \(\operatorname{ht}_R(\mathfrak p)\geq c\). This proves the inequality over \(R\).
Localization at \(\varpi\) preserves the heights of primes not containing \(\varpi\), which proves the
inequality over \(R[1/\varpi]\).
\end{proof}

We shall use without further mention the fact that every \(P\in\Perv(A,\overline{\bbQ}_{\ell})\)
is defined over a finite extension of \(\Ql\) \cite[2.2.18]{BBDG18}.
A continuous character \(\chi\colon\Tl A\to\overline{\bbQ}_{\ell}^{\times}\) takes values in the finite
extension generated by its values on a \(\Zl\)-basis. Taking a compositum, we may
enlarge \(E\) so that both \(P\) and \(\chi\) are defined over \(E\).

We shall also use the following lemma, proved in Appendix~\ref{sec:integral-perverse}.

\begin{lem}\label{lem:perverse-lattice}
Let \(P_E\in\Perv(A,E)\). There exists \(P_{\cO}\in\Perv(A,\cO)\) with
\(P_{\cO}\otimes_{\cO}E\simeq P_E\) such that multiplication by \(\varpi\) is a monomorphism.
Its modular reduction \(\overline P:=P_{\cO}\otimes_{\cO}^{L}\kappa\) is therefore perverse.
\end{lem}

\subsection{Codimension estimates and generic vanishing}
\label{sec:algebraic-model}

\begin{proof}[Proof of Corollary~\ref{cor:D1}]
Choose a finite extension \(E/\Ql\) and \(P_E\in\Perv(A,E)\) with
\(P_E\otimes_E\overline{\bbQ}_{\ell}\simeq P\). By Lemma~\ref{lem:perverse-lattice}, choose a
torsion-free integral form \(P_{\cO}\) of \(P_E\), whose modular reduction
\(\overline P=P_{\cO}\otimes_{\cO}^{L}\kappa\) is perverse.

Let \(\Ldot_R\) be the rank one \(R\)-local system associated to the canonical character and let
\[
  K_{\cO}=R\Gamma\bigl(A,P_{\cO}\otimes_{\cO}^L\Ldot_R\bigr).
\]
This is the integral Fourier--Mellin transform of \cite[Definition~4.4]{EK21}. By
\cite[Proposition~4.1(1)--(2)]{EK21}, it is an object of
\(D^b_{\mathrm{coh}}(R)\), and derived base change identifies
\[
  K_{\cO}\otimes_R^L\overline R
  \ \simeq\
  R\Gamma\bigl(A,\overline P\otimes_{\kappa}^L\Ldot_{\overline R}\bigr)
  \ =\ \widehat{\FM}_A(\overline P),
\]
the completed stalk at the trivial character of the \(\kappa\)-linear Fourier--Mellin transform of
\(\overline P\), formed over \(\overline R=S_{\kappa}\) as in Corollary~\ref{cor:kappa-codim}. The
equality holds by definition, the reductions of \(\Ldot_R\) modulo \(\varpi\) and the powers of the
maximal ideal being the tautological local systems of (\ref{para:tautological}) with
\(\kappa\)-coefficients. Since
\(\overline P\) is perverse, Corollary~\ref{cor:kappa-codim} gives
\[
  \codim_{\Spec\overline R}\Supp
  \rH^i\bigl(K_{\cO}\otimes_R^L\overline R\bigr)\ \geq\ 2i
  \qquad(i\geq0).
\]
Lemma~\ref{lem:special-generic-codim} yields
\[
  \codim_{\Spec R[1/\varpi]}\Supp
  \rH^i\bigl(K_{\cO}[1/\varpi]\bigr)\ \geq\ 2i.
\]

Let \(R_E=R[1/\varpi]\). The map
\[
  R_E\longrightarrow
  \mathfrak R=R_E\otimes_E\overline{\bbQ}_{\ell}
\]
is faithfully
flat and thus flat base change gives
\[
  \rH^i\bigl(K_{\cO}[1/\varpi]\bigr)\otimes_{R_E}\mathfrak R
  \ \simeq\
  \rH^i\bigl(K_{\cO}\otimes_R^L\mathfrak R\bigr).
\]
If \(\mathfrak q\) belongs to the support of the module on the right and
\(\mathfrak p=\mathfrak q\cap R_E\), then \(\mathfrak p\) belongs to the support over \(R_E\).
Going down gives
\[
  \operatorname{ht}_{\mathfrak R}(\mathfrak q)\ \geq\
  \operatorname{ht}_{R_E}(\mathfrak p)\ \geq\ 2i.
\]
By \cite[Definition~4.4]{EK21},
\[
  K_{\cO}\otimes_R^L\mathfrak R
  \ \simeq\
  \mathfrak{FM}_A(P).
\]
By \cite[Proposition~A.2.2.2 and the proof of Proposition~A.2.2.3]{GabberLoeser1996},
\(\mathfrak R\) is a regular Jacobson domain whose completion at each maximal ideal is isomorphic to
\(\overline{\bbQ}_{\ell}[[x_1,\ldots,x_{2g}]]\). Hence every maximal ideal has height \(2g\), and
codimensions computed on \(\Spec\mathfrak R\) agree with those on
\(\operatorname{Spm}(\mathfrak R)\).
\end{proof}

Recall that \(\operatorname{Spm}(\mathfrak R)\) is identified with the group of continuous characters
\(\Tl A\to\overline{\bbQ}_{\ell}^{\times}\) \cite[\S3.1]{EK21}. The following corollary removes the
arithmeticity hypothesis from \cite[Theorem~1.5]{EK21}.

\begin{cor}\label{cor:jumping-loci}
Let \(P\in\Perv(A,\overline{\bbQ}_{\ell})\). For \(\chi\in\operatorname{Spm}(\mathfrak R)\) let
\(L_{\chi}\) be the rank one \(\overline{\bbQ}_{\ell}\)-local system with monodromy
\(\pi_1(A)\twoheadrightarrow\Tl A\xrightarrow{\ \chi\ }\overline{\bbQ}_{\ell}^{\times}\), and let
\[
  \Sigma^{i}(P)\ =\ \bigl\{\chi\in\operatorname{Spm}(\mathfrak R)\ :\
  \rH^{i}\bigl(A,P\otimes L_{\chi}\bigr)\neq0\bigr\}.
\]
Then \(\codim_{\operatorname{Spm}(\mathfrak R)}\Sigma^{i}(P)\geq2|i|\) for every \(i\in\bbZ\).
\end{cor}

\begin{proof}
Let \(i\geq0\) and let \(\chi\in\operatorname{Spm}(\mathfrak R)\), with residue field
\(\kappa(\chi)=\overline{\bbQ}_{\ell}\) \cite[\S3.1]{EK21}. Choose a finite extension \(E/\Ql\) over which
both \(P\) and \(\chi\) are defined, and an integral form \(P_{\cO}\) using
Lemma~\ref{lem:perverse-lattice}. Since \(\chi(\Tl A)\subseteq1+\varpi\cO\), the character induces a surjection
\[
  R=\cO[[\Tl A]]\longrightarrow\cO,\qquad [\gamma]\longmapsto\chi(\gamma).
\]
Applying \cite[Proposition~4.1(2)]{EK21} to the integral transform of \(P_{\cO}\) along this quotient,
then extending scalars from \(\cO\) to \(\overline{\bbQ}_{\ell}\), gives
\begin{equation}\label{eq:ell-adic-character-base-change}
  \mathfrak{FM}_A(P)\otimes^{L}_{\mathfrak R}\kappa(\chi)\ \simeq\
  R\Gamma\bigl(A,P\otimes L_{\chi}\bigr).
\end{equation}
Suppose \(\chi\notin\bigcup_{j\geq i}\Supp\rH^{j}(\mathfrak{FM}_A(P))\), a finite union. Localizing at
\(\chi\) gives \(\mathfrak{FM}_A(P)_{\chi}\in D^{<i}\). Its derived tensor product with
\(\kappa(\chi)\) also lies in \(D^{<i}\), so \(\rH^{i}(A,P\otimes L_{\chi})=0\). Hence
\[
  \Sigma^{i}(P)\ \subseteq\ \bigcup_{j\geq i}\Supp\rH^{j}(\mathfrak{FM}_A(P)),
\]
and Corollary~\ref{cor:D1} gives \(\codim\Sigma^{i}(P)\geq2i\). For \(i<0\), Verdier
duality on the proper \(A\) and the identity
\(\bbD(P\otimes L_{\chi})\simeq\bbD P\otimes L_{\chi^{-1}}\) give
\(\rH^{i}(A,P\otimes L_{\chi})^{\vee}\simeq\rH^{-i}(A,\bbD P\otimes L_{\chi^{-1}})\), so
\(\Sigma^{i}(P)=\operatorname{inv}\bigl(\Sigma^{-i}(\bbD P)\bigr)\), where the inversion
\(\operatorname{inv}\) of \(\operatorname{Spm}(\mathfrak R)\) is induced by a ring automorphism
of \(\mathfrak R\) and preserves codimension. The case \(i\geq0\) applied to \(\bbD P\) thus implies
the result for \(P\) with \(i<0\).

\end{proof}

\begin{proof}[Proof of Corollary~\ref{cor:D3}(2)]
Since \(\mathfrak R\) is regular, \(\mathfrak{FM}_A(P)\) is perfect. By
\eqref{eq:ell-adic-character-base-change}, its fibre at \(\chi\) is \(R\Gamma(A,P\otimes L_{\chi})\).
The Euler characteristic of the fibres is locally constant, hence constant because \(\mathfrak R\)
is a domain, and equals \(\chi(A,P)\) at the trivial character.

Let \(Z=\bigcup_{i\neq0}\overline{\Sigma^i(P)}\), with closures taken in
\(\operatorname{Spm}(\mathfrak R)\). This is a finite union, and Corollary~\ref{cor:jumping-loci}
gives \(\codim Z\geq2\). For \(\chi\notin Z\), only \(\rH^0(A,P\otimes L_{\chi})\) can be nonzero,
and its dimension is \(\chi(A,P)\geq0\).
\end{proof}

We now state and prove an \(\ell\)-adic form of Theorem~\ref{thm:A}.

\begin{cor}\label{cor:ell-adic-A}
Let \(P\in\Perv(A,\overline{\bbQ}_{\ell})\). For every positive integer \(m\) invertible in \(k\)
and every \(i\in\bbZ\),
\[
  \dim_{\overline{\bbQ}_{\ell}}\rH^{i}\bigl(A,[m]^{*}P\bigr)\ \leq\ C_{P}\,m^{2(g-|i|)},
\]
where \(C_{P}\) is a constant independent of \(m\).
\end{cor}

\begin{proof}
Choose a finite extension \(E/\Ql\) and \(P_E\in\Perv(A,E)\) with
\(P_E\otimes_E\overline{\bbQ}_{\ell}\simeq P\). By Lemma~\ref{lem:perverse-lattice}, choose a
torsion-free integral form \(P_{\cO}\) of \(P_E\). Its modular reduction
\(\overline P=P_{\cO}\otimes^{L}_{\cO}\kappa\) is perverse. For \(m\geq1\) invertible in \(k\), let
\(K=R\Gamma(A,[m]^{*}P_{\cO})\). By \cite[Theorems~6.3(iii) and~7.2(i)]{Eke90}, applied with the
coefficient ring \(\cO\) as in the proof of Lemma~\ref{lem:FM-finiteness}, \(K\in D^b_{\mathrm{coh}}(\cO)\),
and \(K\) is perfect since \(\cO\) is regular.

Since \([m]^{*}\) and
\(R\Gamma(A,-)\) commute with the coefficient changes \(\otimes^{L}_{\cO}E\) and
\(\otimes^{L}_{\cO}\kappa\), one has \(K\otimes^{L}_{\cO}E\simeq R\Gamma(A,[m]^{*}P_{E})\) and
\(K\otimes^{L}_{\cO}\kappa\simeq R\Gamma(A,[m]^{*}\overline P)\).
The reduction triangle gives an injection
\(\rH^{i}(K)\otimes_{\cO}\kappa\hookrightarrow\rH^{i}(K\otimes^{L}_{\cO}\kappa)\), so
\[
  \dim_{E}\rH^{i}\bigl(A,[m]^{*}P_{E}\bigr)\ =\ \operatorname{rk}_{\cO}\rH^{i}(K)
  \ \leq\ \dim_{\kappa}\rH^{i}\bigl(A,[m]^{*}\overline P\bigr).
\]
Regard \(\overline P\) as a perverse sheaf with \(\Fl\)-coefficients. Its cohomology dimensions
are multiplied by \([\kappa:\Fl]\), which is independent of \(m\). Applying Theorem~\ref{thm:A}
and dividing by this fixed factor bounds the right-hand side by \(C\,m^{2(g-|i|)}\). Finally,
\(\dim_{\overline{\bbQ}_{\ell}}\rH^{i}(A,[m]^{*}P)=\dim_{E}\rH^{i}(A,[m]^{*}P_{E})\).
\end{proof}

The preceding Betti bounds also bound the number of characters with nonzero twisted cohomology.
For a positive integer \(m\) invertible in \(k\), a field \(\Lambda\), and a character
\(\chi\colon A[m]\to\Lambda^{\times}\), let \(L_{\chi}\) be the rank one \(\Lambda\)-local system on
\(A\) with monodromy \(\pi_{1}(A)\twoheadrightarrow A[m]\xrightarrow{\ \chi\ }\Lambda^{\times}\).

\begin{cor}\label{cor:character-count}
Let \(m\geq1\) be invertible in \(k\) and let \(i\in\bbZ\).
\begin{enumerate}
\item For \(P\in\Perv(A,\overline{\bbQ}_{\ell})\),
\[
  \#\bigl\{\chi\in\Hom(A[m],\overline{\bbQ}_{\ell}^{\times})\ :\ \rH^{i}(A,P\otimes L_{\chi})\neq0\bigr\}
  \ \leq\ \dim_{\overline{\bbQ}_{\ell}}\rH^{i}\bigl(A,[m]^{*}P\bigr)\ \leq\ C_{P}\,m^{2(g-|i|)}.
\]
\item Let \(m\) be prime to \(\ell\), let \(\kappa\) be a finite extension of \(\Fl\) containing the
\(m\)-th roots of unity, and let \(P\in\Perv(A,\kappa)\). Then
\[
  \#\bigl\{\chi\in\Hom(A[m],\kappa^{\times})\ :\ \rH^{i}(A,P\otimes_{\kappa}L_{\chi})\neq0\bigr\}
  \ \leq\ \dim_{\kappa}\rH^{i}\bigl(A,[m]^{*}P\bigr)\ \leq\ C_{P}\,m^{2(g-|i|)}.
\]
\end{enumerate}
\end{cor}

\begin{proof}
Let \(\Lambda\) be \(\overline{\bbQ}_{\ell}\) in (1) and \(\kappa\) in (2). In both cases \(m\) is
invertible in \(\Lambda\) and \(\Lambda\) contains the \(m\)-th roots of unity, so the characters
\(\chi\in\Hom(A[m],\Lambda^{\times})\) identify \(\Lambda[A[m]]\) with the product of \(m^{2g}\)
copies of \(\Lambda\). The \'etale \([m]\) satisfies
\([m]_{*}[m]^{*}P\simeq P\otimes_{\Lambda}[m]_{*}\Lambda_{A}\) by the projection formula, and
\([m]_{*}\Lambda_{A}\) is the rank one \(\Lambda[A[m]]\)-local system of the torsor \([m]\).
The character decomposition gives \([m]_{*}\Lambda_{A}\simeq\bigoplus_{\chi}L_{\chi}\). Taking cohomology,
\[
  \rH^{i}\bigl(A,[m]^{*}P\bigr)\ \simeq\ \bigoplus_{\chi\in\Hom(A[m],\Lambda^{\times})}
  \rH^{i}\bigl(A,P\otimes_{\Lambda}L_{\chi}\bigr),
\]
which gives the first inequality in both cases. For the second, apply Corollary~\ref{cor:ell-adic-A}
in (1), and in (2) apply Theorem~\ref{thm:A} to \(P\) regarded as a perverse sheaf with
\(\Fl\)-coefficients.
\end{proof}

For \(i\neq0\) the characters of \(A[m]\) with \(\rH^{i}(A,P\otimes L_{\chi})\neq0\) thus have
density at most \(C_{P}m^{-2|i|}\) among all \(m^{2g}\) characters. This is a counting analogue of
the codimension bound \(2|i|\) of Corollary~\ref{cor:jumping-loci}, and it is available for
\(\Fl\)-coefficients.

\begin{rem}\label{cor:betti-torus}
Let \(A\) be an abelian variety over \(\bbC\), let \(\kappa\) be a finite field of characteristic
\(\ell\), and let \(P\in\Perv(A,\kappa)\). The codimension estimates of
\cite[Theorems~3.1 and~3.6]{BhattSchnellScholze2018} hold for the Betti Fourier--Mellin transform
of the analytification of \(P\) on \(\Spec\kappa[\pi_1(A(\bbC),0)]\). The supports of its \(i\)-th
cohomology sheaves have codimension at least \(2i\) for \(i\geq0\), and its cohomology jump loci
in degree \(i\) have codimension at least \(2|i|\) for every \(i\in\bbZ\).

Indeed, the character at every closed point has finite order. After extension to its residue field
\(\kappa'\), the completion of the Betti transform at this point is \(\widehat{\FM}_A(P')\) over
\(\kappa'[[\Tl A]]\), where \(P'\) is obtained from \(P\) by extension of coefficients and twisting
by the associated \'etale local system. This follows by comparison at each finite quotient
\cite[Expos\'e~XI, Th\'eor\`eme~4.4(i) and Expos\'e~XVI, Th\'eor\`eme~4.1]{ArtinGrothendieckSGA4}.
Corollary~\ref{cor:kappa-codim} therefore gives the support estimates, since codimension can be
checked after completion at closed points. The estimates for the jump loci follow from the fibre
formula \cite[Lemma~2.6]{BhattSchnellScholze2018} and Verdier duality, as in the proof of
Corollary~\ref{cor:jumping-loci}.
\end{rem}

\begin{rem}\label{rem:integral-uniform}
With integral coefficients the uniform vanishing of Theorem~\ref{thm:B} becomes a uniform divisibility. Let
\(P_{\cO}\in\Perv(A,\cO)\) be an integral form as in Lemma~\ref{lem:perverse-lattice}, with perverse
reduction \(\overline P\), regarded as a perverse sheaf with \(\Fl\)-coefficients. By
Corollary~\ref{cor:connectivity}(1) for \(\bbD\overline P\) and Lemma~\ref{lem:tower-vanishing}, there is
\(e\geq0\) such that \([\ell^{e}]^{*}\) is zero on \(\rH^{i}(A,[\ell^{n}]^{*}\overline P)\) for all \(n\geq0\)
and \(i>0\). The long exact sequence of the triangle \(P_{\cO}\xrightarrow{\varpi}P_{\cO}\to\overline P\)
gives injections \(\rH^{i}(A,[m]^{*}P_{\cO})\otimes_{\cO}\kappa\hookrightarrow\rH^{i}(A,[m]^{*}\overline P)\)
compatible with the pullback maps. For \(x\in\rH^{i}(A,[\ell^{n}]^{*}P_{\cO})\) with \(i>0\), the image of
\([\ell^{e}]^{*}x\) in \(\rH^{i}(A,[\ell^{n+e}]^{*}\overline P)\) vanishes, so
\([\ell^{e}]^{*}x\in\varpi\rH^{i}(A,[\ell^{n+e}]^{*}P_{\cO})\), and by induction on \(r\)
\[
  [\ell^{re}]^{*}\,\rH^{i}\bigl(A,[\ell^{n}]^{*}P_{\cO}\bigr)\ \subseteq\
  \varpi^{r}\,\rH^{i}\bigl(A,[\ell^{n+re}]^{*}P_{\cO}\bigr)
  \qquad\text{for all } n,r\geq0 \text{ and } i>0.
\]
\end{rem}

\appendix
\section{Integral perverse sheaves}
\label{sec:integral-perverse}

Let \(E/\Ql\) be finite, with ring of integers \(\cO\), uniformizer \(\varpi\) and residue field
\(\kappa\), and let \(X\) be a separated scheme of finite type over \(k\). Constructible \(\cO\)-coefficients
are those of the adic formalism of \cite[2.2.14 and 2.2.18]{BBDG18}, which applies over the
algebraically closed field \(k\). In this formalism
\(\Dbc(X,E)=\Dbc(X,\cO)\otimes_{\cO}E\), the category
\(\Dbc(X,\overline{\bbQ}_{\ell})\) is the filtered \(2\)-colimit of the categories \(\Dbc(X,E)\),
and the perverse \(t\)-structures are compatible with these constructions.

Over \(\cO\), the
perverse \(t\)-structure is the one defined by the support and cosupport conditions of the middle
perversity, the \(p\)-structure of \cite[\S2.1]{Jut09}. Its dual \(p_+\)-structure is not used in this
article. In particular every \(P\in\Perv(X,\overline{\bbQ}_{\ell})\) arises from some
\(P_E\in\Perv(X,E)\), and every \(P_E\) arises from some \(Q\in\Dbc(X,\cO)\) with
\(Q\otimes_{\cO}^{L}E\simeq P_E\) \cite[2.2.18]{BBDG18}.

Following \cite[Definition~2.10]{Jut09}, an object \(N\in\Perv(X,\cO)\) is \emph{torsion} if
\(\varpi^{\nu}\cdot\mathrm{id}_N=0\) for some \(\nu\), and \emph{torsion-free} if
\(\varpi\cdot\mathrm{id}_N\) is a monomorphism. The next lemma gives the torsion subobject needed
to construct the torsion-free integral form in Lemma~\ref{lem:perverse-lattice}.

\begin{lem}\label{lem:torsion-lattice}
Let \(X\) be a separated scheme of finite type over \(k\) and let \(N\in\Perv(X,\cO)\).
\begin{enumerate}
\item If \(N\) is torsion-free, then \(N\otimes^{L}_{\cO}\kappa\) lies in \(\Perv(X,\kappa)\) and equals
the cokernel of \(\varpi\cdot\mathrm{id}_{N}\) in \(\Perv(X,\cO)\).
\item The ascending chain of perverse subobjects \(\ker(\varpi^{\nu}\cdot\mathrm{id}_{N})\)
stabilizes. Consequently \(N\) has a greatest torsion subobject \(T\subseteq N\), the quotient
\(N/T\) is torsion-free, and \((N/T)\otimes^{L}_{\cO}E\simeq N\otimes^{L}_{\cO}E\).
\end{enumerate}
\end{lem}

\begin{proof}
(1) The triangle \(N\xrightarrow{\varpi}N\to N\otimes^{L}_{\cO}\kappa\) gives the exact sequence of
perverse cohomology
\[
  0\to{}^{p}\!H^{-1}\bigl(N\otimes^{L}_{\cO}\kappa\bigr)\to N\xrightarrow{\ \varpi\ }N
  \to{}^{p}\!H^{0}\bigl(N\otimes^{L}_{\cO}\kappa\bigr)\to0.
\]
If \(\varpi\cdot\mathrm{id}_{N}\) is a monomorphism, the first term vanishes, so the reduction is
perverse in \(\Dbc(X,\cO)\) and equals the cokernel. The complex \(N\otimes^{L}_{\cO}\kappa\) is
naturally an object of \(\Dbc(X,\kappa)\), and restriction of scalars along \(\cO\to\kappa\) commutes
with \(i_{x}^{*}\) and \(i_{x}^{!}\) and preserves the vanishing of cohomology sheaves. Since the
perverse \(t\)-structures of \(\Dbc(X,\cO)\) and \(\Dbc(X,\kappa)\) are both defined by the vanishing of
\(\cH^{n}i_{x}^{*}\) for \(n>-\dim\overline{\{x\}}\) and of \(\cH^{n}i_{x}^{!}\) for
\(n<-\dim\overline{\{x\}}\) at the points \(x\) of \(X\) \cite[\S2.1]{Jut09}, it follows that
\(N\otimes^{L}_{\cO}\kappa\in\Perv(X,\kappa)\).

(2) The category \(\Perv(X,\cO)\) is noetherian \cite[4.0(b)]{BBDG18}, so the ascending chain
\(K_{\nu}:=\ker(\varpi^{\nu}\cdot\mathrm{id}_{N})\) stabilizes. For large \(\nu\), \(T:=K_{\nu}\) contains
every torsion subobject, the kernel of \(\varpi\cdot\mathrm{id}_{N/T}\) is \(K_{\nu+1}/K_{\nu}=0\), and
\(T\otimes^{L}_{\cO}E=0\).
\end{proof}

\begin{proof}[Proof of Lemma~\ref{lem:perverse-lattice}]
Choose \(Q\in\Dbc(A,\cO)\) with \(Q\otimes_{\cO}^{L}E\simeq P_E\). Scalar extension to \(E\)
is perverse \(t\)-exact, so \(N={}^p\!H^0(Q)\) has generic fibre \(P_E\). By
Lemma~\ref{lem:torsion-lattice}(2), \(N\) has a greatest torsion subobject \(T\), and
\(P_{\cO}=N/T\) is torsion-free with \(P_{\cO}\otimes_{\cO}E\simeq P_E\). Its derived reduction
\(P_{\cO}\otimes_{\cO}^{L}\kappa\) is perverse by Lemma~\ref{lem:torsion-lattice}(1).
\end{proof}

\bibliographystyle{alpha-custom}
\bibliography{bib/references}

\end{document}